\documentclass[11pt,reqno]{amsart}
\usepackage{amsmath}
\usepackage{amsthm}
\usepackage{cite}
\usepackage{amssymb}
\usepackage{graphicx}
\usepackage{amsfonts}
\usepackage{latexsym}
\usepackage{setspace}
\usepackage{hyperref}
\usepackage{multicol}
\usepackage[active]{srcltx}
\usepackage{mathrsfs}
\usepackage{palatino}

\usepackage{lineno}

\newcommand{\ov}{\overline}

\newcommand{\C}{ \mathbb{C}}
\newcommand{\D}{ \mathbb{D}}

\newcommand{\Z}{ \mathbb{Z}}

\newcommand{\R}{ \mathbb{R}}

\newcommand{\ran}{\operatorname{ran}}

\newcommand{\norm}[1]{\| #1 \|}
\newcommand{\inner}[1]{\langle #1 \rangle}

\newcommand{\N}{\mathbb{N}}

\newcommand{\T}{\mathbb{T}}
\newcommand{\h}{\mathcal{H}}
\newcommand{\HH}{\mathscr{H}}
\newcommand{\K}{\mathcal{K}}
\newcommand{\KK}{\mathscr{K}}
\newcommand{\minimatrix}[4]{\begin{bmatrix} #1 & #2 \\ #3 & #4 \end{bmatrix}  }

\newcommand{\rank}{\operatorname{rank}}
\renewcommand{\vec}[1]{{\bf #1}}

\renewcommand{\phi}{\varphi}

\numberwithin{equation}{section}
\theoremstyle{plain}
\newtheorem{Theorem}[equation]{Theorem}
\newtheorem{Proposition}[equation]{Proposition}
\newtheorem{Corollary}[equation]{Corollary}

\theoremstyle{definition}
\newtheorem{Definition}[equation]{Definition}

\newtheorem{Example}[equation]{Example}
\newtheorem{Remark}[equation]{Remark}

\allowdisplaybreaks

\begin{document}
\bibliographystyle{amsplain}

    \title{Model spaces: A survey}

    \author{Stephan Ramon Garcia}
    \address{   Department of Mathematics\\
            Pomona College\\
            Claremont, California\\
            91711 \\ USA}
    \email{Stephan.Garcia@pomona.edu}
    \urladdr{http://pages.pomona.edu/\textasciitilde sg064747}
    \thanks{First author partially supported by National Science Foundation Grant DMS-1265973.}

    \author{William T. Ross}
    \address{   Department of Mathematics and Computer Science\\
            University of Richmond\\
            Richmond, Virginia\\
            23173 \\ USA}
    \email{wross@richmond.edu}
    \urladdr{http://facultystaff.richmond.edu/~wross}

    \begin{abstract}
	This is a brief and gentle introduction, aimed at graduate students, to the subject of model subspaces of the Hardy space. 
    \end{abstract}

\maketitle

\section{Introduction} In the summer of 2013, the authors gave a series of lectures and minicourses in Montr\'eal, Lens, and Helsinki on the topic of model spaces.  In preparing for these lectures, we discovered the need for an easy introduction to model spaces suitable for the graduate students
who formed the intended audience for our lectures. 
The standard texts on the subject \cite{N1, N2, N3, Bercovici} are thorough and encyclopedic, but are sufficiently intimidating
that the beginner might find this whole beautiful subject out of reach. The purpose of this survey is to give the novice a 
friendly, albeit incomplete, introduction to model spaces.

Model spaces are Hilbert spaces of the form $(u H^2)^{\perp}$, where $u$ is an inner function, 
$H^2$ is the classical Hardy space on the open unit disk $\D$, and $\perp$ denotes the orthogonal complement in $H^2$. 
On a functional analysis level, model spaces are the orthogonal complements of the nontrivial invariant subspaces of the 
unilateral shift $S f = z f$ on $H^2$.  These subspaces were characterized as $u H^2$ by Beurling in his famous 1949 paper \cite{B}. 
As such, the spaces $(u H^2)^{\perp}$ are the invariant subspaces of the backward shift operator $S^{*} f = (f - f(0))/z$ on $H^2$. 
However, unlike the spaces $u H^2$ which are simple to understand (i.e., all $H^2$ multiples of the inner function $u$), the model spaces $(u H^2)^{\perp}$ are much more 
troublesome.  For instance, it is not immediately clear which functions actually belong $(u H^2)^{\perp}$ or what properties these functions have. 

A major breakthrough in the study of model spaces occurred in 1970, with the publication of the seminal paper of Douglas, Shapiro, and Shields \cite{DSS}.  Extending some partial results of earlier authors, they showed that functions in $(u H^2)^{\perp}$ have analytic continuations in the same neighborhoods of points on the unit circle as does $u$. However, a generic inner function $u$ need not have an analytic continuation across any point of the unit circle $\T$ and thus a new type of continuation was needed. Fortunately this type of continuation, called a \emph{pseudocontinuation}, was discovered and formalized in two earlier papers \cite{MR0241614, MR0267080} of Shapiro and indeed turned out to be the determining characterization of functions in $(u H^2)^{\perp}$. In fact, a notion of continuation more general than analytic continuation was already being discussed, in various forms (rational approximation for example), often never actually being called a ``continuation.'' For example an earlier paper of Tumarkin \cite{Tum1} discussed functions in $(u H^2)^{\perp}$ via controlled approximation by rational functions.
 
From these beginnings, the function theoretic aspects of model spaces have truly grown. We now understand much of the subtle relationship between the boundary behavior of functions in $(u H^2)^{\perp}$ and the angular derivative of $u$ through the papers of Ahern and Clark \cite{AC70, AC70a}.

On an operator theory level, work of Sz.-Nagy and Foias showed that the compression of the unilateral shift $S$ to a model space $(u H^2)^{\perp}$ is a model for a certain class of contractions \cite{MR2760647}. This seminal work has been studied by many people and continues to have relevance in operator theory \cite{RR}.  Clark later examined unitary rank-one perturbations of this compressed shift and was able to come up with an exact spectral realization of this unitary operator using a family of measures that now bear his name \cite{MR0301534}. These measures enjoy truly fascinating properties that have been harnessed by Aleksandrov \cite{MR1039571, MR1734326}, Poltoratski \cite{MR1223178}, Sarason \cite{MR2198367, MR1289670, Sarason}, 
and others to explore fine boundary properties of the inner function $u$ as well as completeness of families of reproducing kernels in $(u H^2)^{\perp}$. These measures have even appeared in a somewhat different form in mathematical physics \cite{Simon-Wolff, Simon}.

Model spaces have also been helpful is examining nearly invariant subspaces (subspaces that contain $f/z$ whenever they contain a function $f$ satisfying $f(0) = 0$) \cite{Sa88}. These nearly invariant subspaces have been useful in developing characterizations of the invariant subspaces for the shift $f \mapsto z f$ on Hardy spaces of planar domains with holes or slits \cite{MR2548414, MR1371943, Hi88}.

Though model spaces have many connections to old ideas in analysis (analytic continuation, factorization, pseudocontinuation, etc.), they continue to be relevant. Recent work of Makarov and Poltoratski \cite{MP05} show that a spectral realization of certain selfadjoint Schr{\" o}dinger operators can be realized through model spaces and Clark measures. 
Model spaces also make important connections to the subject of complex symmetric operators \cite{CCO, G-P, G-P-II, SNCSO, CSPI}, a certain class of Hilbert space operators that are frequently modeled by compressions of Toeplitz operators to model spaces \cite{GR-TTO, STZ, TTOSIUES}.

As the subject of model spaces is quite vast, we again emphasize that these notes are not meant to be an encyclopedic treatment. These notes are instead meant to give the beginning student a reason to want to study this material and to provide the means for them to take their first few steps into this rich and fertile territory. 

\section{Preliminaries}
	Before proceeding, we start with a brief review of Hardy space theory. 
	The material presented in this section is now considered classical and can be found in many standard texts 
	\cite{Duren, Garnett, Hoffman, Koosis}.
	A relatively new text that might be more suitable for a student who is new to Hardy spaces is \cite{MR2270722}. 
	
\subsection{Lebesgue spaces}
	Let $m$ denote normalized Lebesgue measure on the unit circle $\T$ (i.e., $m = d \theta/2\pi$) and let 
	$L^2 := L^2(\T, m)$ denote the space of $m$-measurable (i.e., Lebesgue measurable) functions $f: \T \to \C$ such that
	\begin{equation*}
		\norm{f} := \left(\int_{\T} |f(\zeta)|^2 dm(\zeta) \right)^{\frac{1}{2}}
	\end{equation*}
	is finite.  As such, $L^2$ is a Hilbert space endowed with the inner product
	\begin{equation*}
		\langle f, g \rangle := \int_{\T} f(\zeta) \ov{g(\zeta)} \,dm(\zeta).
	\end{equation*}
	A simple calculation using the fact that $m(\T) = 1$ 
	shows that the family of functions $\{\zeta \mapsto \zeta^n: n \in \Z\}$ is an orthonormal basis for $L^2$. 
	The coefficients 
	\begin{equation*}
		\widehat{f}(n) := \langle f, \zeta^n\rangle = \int_{\T} f(\zeta) \overline{\zeta}^n dm(\zeta)
	\end{equation*}
	of a function $f$ in $L^2$ with respect to this basis are called the (complex) \emph{Fourier coefficients} of $f$.  In light of Parseval's Identity
	\begin{equation*}
		\norm{f}^2 = \sum_{n \in \Z} |\widehat{f}(n)|^2,
	\end{equation*}
	we see that the $L^2$ norm of $f$ coincides with the norm of the sequence $\{\widehat{f}(n): n \in \Z\}$
	of Fourier coefficients in the space $\ell^2(\Z)$ of all square-summable sequences on $\Z$.
	We therefore identify the Hilbert spaces $L^2$ and $\ell^2(\Z)$ via $f \leftrightarrow \{\widehat{f}(n): n \in \Z\}$. 	

	We also require the space $L^{\infty}:= L^{\infty}(\T)$ of all essentially bounded functions
	on $\T$ which, when equipped with the norm
	\begin{equation*}
		\norm{f}_{\infty} := \operatorname{ess-sup}_{\zeta \in \T} |f(\zeta)|,
	\end{equation*}
	becomes a Banach algebra.  We also remark that for any $\phi$ in $L^{\infty}$, the multiplication operator $f \mapsto \phi f$
	on $L^2$ is bounded and has operator norm equal to $\norm{\phi}_{\infty}$.

\subsection{Hardy spaces}
		
	For an analytic function $f$ on $\D$ the integral means
	\begin{equation*}
		\int_{\T} |f(r\zeta)|^2\,dm(\zeta)
	\end{equation*}
	are increasing as a function of $r$ on $(0, 1)$. Indeed, if $f(z) = \sum_{n=0}^{\infty} a_n z^n$, then 
	$$\int_{\T} |f(r \zeta)|^2 dm(\zeta) = \sum_{n=0}^{\infty} |a_n|^2 r^{2 n},$$
	which is clearly increasing in $r$.
	This leads us to define 
	the \emph{Hardy space} $H^2$ as those $f$ for which
	\begin{equation}\label{eq:HpNorm}
		\norm{f} :=  \lim_{r\to 1^-} \left(\int_{\T} |f(r\zeta)|^2\,dm(\zeta)
		\right)^{\frac{1}{2}}
	\end{equation}
	is finite. 
	It is no accident that we use $\|\cdot\|$ to denote both the norm in $L^2$ and in $H^2$. 
	Indeed, classical work of Fatou and Riesz show that the {\em radial 
	limit}\footnote{It turns out that each $f$ in $H^2$ has a finite {\em non-tangential limit} 
		at $\zeta$ for almost every $\zeta$. By this we mean that the limit of $f(z)$ exists as $z$ approaches 
		$\zeta$ in every \emph{Stolz region} $\{z \in \D: |z - \zeta| < \alpha (1 - |z|)\}$, $\alpha > 1$.}
	\begin{equation}\label{boundary-f}
		f(\zeta) := \lim_{r \to 1^{-}} f(r \zeta)
	\end{equation}
	exists for $m$-almost every $\zeta$ in $\T$ and 
	it turns out that the $H^2$ norm of $f$ and the $L^2$ norm of its boundary function, defined in \eqref{boundary-f}, coincide. 
	In this manner, we often regard $H^2$ as a closed subspace of $L^2$.  As such, $H^2$ inherits the structure
	of a Hilbert space from $L^2$.
	
	Let $H^{\infty}$ denote the space of all bounded analytic functions on $\D$,
	endowed with the supremum norm
	\begin{equation*}
		\norm{f}_{\infty} := \sup_{z \in \D} |f(z)|.
	\end{equation*}
	In a similar manner, the radial boundary function of an $H^{\infty}$ function belongs to $L^{\infty}$ and one has a corresponding
	maximum modulus type result
	\begin{equation*}
		\sup_{z \in \D} |f(z)| = \operatorname{ess-sup}_{\zeta \in \T} |f(\zeta)|,
	\end{equation*}
	that allows us to view $H^{\infty}$ as a Banach subalgebra of $L^{\infty}$.  Moreover, $H^{\infty}$ happens to be the multiplier
	algebra for $H^2$, meaning that the operator of multiplication by an analytic function $\phi$
	on $H^2$ (i.e., $f \mapsto \phi f$) is bounded if and only if $\phi$ belongs to $H^{\infty}$. 
	The norm of this multiplication operator is precisely $\|\phi\|_{\infty}$.
	
	The inner product on the Hardy space $H^2$ is given by
	\begin{equation*}
		\inner{f,g} =  
		\int_{\T} f(\zeta)\overline{g(\zeta)}\,dm(\zeta) = \sum_{n=0}^{\infty} a_n \overline{b_n},
	\end{equation*}
	where $f(z) = \sum_{n=0}^{\infty} a_n z^n$ and $g(z) = \sum_{n=0}^{\infty} b_n z^n$ denote typical elements of $H^2$.
	In other words, we have a natural identification of $H^2$ with the sequence space $\ell^2(\N)$, where
	each $f$ in $H^2$ is identified with its sequence of Taylor coefficients $\{a_n\}_{n \geqslant 0}$.
		
	Of great importance is the manner in which $H^2$ sits inside of $L^2$.  
	If $f(z) = \sum_{n=0}^{\infty} a_n z^n$ belongs to $H^2$, 
	then the almost everywhere defined boundary function $f$ has an associated Fourier series 
	\begin{equation*}
		f \sim \sum_{n=0}^{\infty}a_n \zeta^n
	\end{equation*}
	which belongs to the first component in the direct sum 
	\begin{equation*}
		L^2 = H^2 \oplus \overline{zH^2},
	\end{equation*}
	where 
	\begin{equation*}
		\overline{zH^2} = \{ \overline{zh} : h \in H^2\}.
	\end{equation*}
	In terms of Fourier coefficients,
	\begin{equation*}
		H^2 = \bigvee \{ z^n : n \geqslant 0\}, \qquad \overline{zH^2} = \bigvee\{ z^n : n \leqslant -1\},
	\end{equation*}
	where $\bigvee$ denotes the closed linear span in $L^2$. 
	In particular, note that 
	\begin{equation*}
		a_n = 
		\begin{cases}
			\widehat{f}(n) & \text{if $n \geqslant 0$}, \\
			0 & \text{if $n < 0$},
		\end{cases}
	\end{equation*}
	and that the polynomials in $z$ are dense in $H^2$.
	These results are summarized in the following diagram:	
	\begin{equation*}
		\large
		\begin{array}{ccc}
			H^2 & \overset{\text{\tiny identified}}{\longleftarrow\!\longrightarrow} & \ell^2(\N) \\[5pt]
			\bigcap & & \bigcap \\[8pt]
			L^2 & \overset{\text{\tiny identified}}{\longleftarrow\!\longrightarrow} & \ell^2(\Z)
		\end{array}
	\end{equation*}
	
\subsection{The Cauchy-Szeg\H{o} kernel}\label{Subsection:CSZK}			
	In light of the inequality
	\begin{equation*}
		|f(\lambda)| 
		 \leqslant \sum_{n=0}^{\infty} |a_n| |\lambda|^n 
		 \leqslant \left(\sum_{n=0}^{\infty} |a_n|^2\right)^{\frac{1}{2}} \left(\sum_{n=0}^{\infty} |\lambda|^{2n}\right)^{\frac{1}{2}}
		 = \frac{\norm{f}}{\sqrt{1 - |\lambda|^2}},
	\end{equation*}
	which holds for all $\lambda$ in $\D$ and all $f$ in $H^2$, it follows that for fixed $\lambda \in \D$ the point evaluation
	functionals $f \mapsto f(\lambda)$ are bounded on $H^2$ and hence, by the Riesz Representation Theorem, must be of the form
	\begin{equation}\label{eq:CSzK}
		f(\lambda) =\inner{ f, c_{\lambda}}
	\end{equation}
	for some $c_{\lambda}$ in $H^2$.  In fact, it is not hard to show that
	\begin{equation} \label{szego-k}
		c_{\lambda}(z) = \frac{1}{1 - \overline{\lambda}z}, \quad \lambda \in \D,
	\end{equation}
	which is called the \emph{Cauchy-Szeg\H{o} kernel} or, perhaps
	more frequently, the \emph{Cauchy kernel}.  In more general terms, one says that the Cauchy-Szeg\H{o} kernel
	is the \emph{reproducing kernel} for $H^2$.  Of great interest to us are other reproducing kernels, which satisfy equations
	analogous to \eqref{eq:CSzK} on various subspaces of $H^2$ (see Section \ref{Section:RK}).	
	
	Before proceeding, we should also remark that the reproducing formula \eqref{eq:CSzK} 
	is simply a restatement of the identity
	\begin{equation*}
		f(\lambda) = \sum_{n=0}^{\infty} a_n \lambda^n 
		= \big< (a_0,a_1,a_2,\ldots), (1 , \overline{\lambda}, \overline{\lambda}^2,\ldots) \big>_{\ell^2(\N)}.
	\end{equation*}
	When written as a contour integral, \eqref{eq:CSzK} reduces to the Cauchy Integral Formula 
	\begin{equation*}
		f(\lambda) = \int_{\T} \frac{f(\zeta)}{1 - \lambda \overline{\zeta}}\,dm(\zeta)
	\end{equation*}
	for $H^2$ functions (recall from \eqref{boundary-f} that for each $f$ in $H^2$, 
	the almost everywhere defined boundary function $\zeta \mapsto f(\zeta)$ belongs to $L^2$, 
	allowing the preceding integral to be well-defined). 
	Along similar lines, for fixed $n \geqslant 0$ we have
	\begin{equation*}
		f^{(n)}(\lambda) = \big< f, c_{\lambda}^{(n)} \big>,
	\end{equation*}
	where
	\begin{equation*}
		 c_{\lambda}^{(n)}(z) = \frac{n! \overline{\lambda}^{n}}{(1 - \overline{\lambda} z)^{n + 1}}
	\end{equation*}
	is the $n$th derivative of $c_{\lambda}$ with respect to the variable $z$.

	\begin{Proposition}\label{Prop:CKLI}
		The set $\{c_{\lambda} : \lambda \in \D\}$ is linearly independent.
	\end{Proposition}
		
	\begin{proof}
		If $\lambda_1,\lambda_2,\ldots,\lambda_n$ are distinct elements of $\D$ and $$\sum_{i=1}^n \alpha_i c_{\lambda_i} = 0,$$
		then $$\sum_{i=1}^n \overline{\alpha_i} f(\lambda_i) = 0$$ holds for all $f$ in $H^2$.
		The Lagrange Interpolation Theorem provides us with a polynomial $p(z)$ satisfying $p(\lambda_i) = \alpha_i$
		so that $\sum_{i=1}^n |\alpha_i|^2 = 0$.  Thus $\alpha_i = 0$ for $i=1,2,\ldots,n$.
	\end{proof}

\subsection{Canonical factorization}
	Although $H^2$ is a linear space, it is its multiplicative structure that reveals
	its true function-theoretic depth.  We recall here the main ingredients necessary to describe
	the canonical factorization of $H^2$ functions.
	
	\begin{Definition}
		An \emph{inner function} is a bounded analytic function $u$ on $\D$ such that $|u(\zeta)|=1$ for almost every $\zeta$ in $\T$.\footnote{The 
			reader is reminded that whenever we use the term ``boundary function'' or write 
			$f(\zeta)$ for $f \in H^2$ and $\zeta \in \T$, we are referring to the almost everywhere defined radial (non-tangential) limit in \eqref{boundary-f}.}
	\end{Definition}

	The simplest nontrivial example of an inner function is a M\"obius transformation of the form
	\begin{equation*}
		e^{i\theta}\frac{w - z}{1 - \overline{w} z},
	\end{equation*}
	where $|w| < 1$ and $0\leqslant \theta < 2\pi$, which is easily seen to be an automorphism of $\D$ mapping $\T$ onto $\T$.  More generally, 
	if $\{z_n\}_{n \geqslant 1}$ is a sequence of points in $\D\backslash\{0\}$, repeated according to multiplicity, then
	the \emph{Blaschke condition}	
	\begin{equation}\label{eq:BlaschkeCondition}
		\sum_{n = 1}^{\infty} (1 - |z_n|) < \infty
	\end{equation}
	is necessary and sufficient for the convergence (uniformly on compact subsets of $\D$) 
	of the corresponding \emph{Blaschke product}
	\begin{equation}\label{eq:BlaschkeProduct}
		B(z) := z^m \prod_{n = 1}^{\infty} \frac{|z_n|}{z_n} \frac{z_n - z}{1 - \overline{z_n} z},
	\end{equation}
	where $m$ denotes a nonnegative integer.  
	With some work, one can show that every Blaschke product is an inner function \cite[Thm.~2.4]{Duren} (the only thing left to check is that the boundary function is unimodular almost everywhere).  
	The importance of these functions stems from the fact that the Blaschke condition \eqref{eq:BlaschkeCondition} 
	completely characterizes the zero sets for $H^2$ functions.  
	
	\begin{Theorem}
		A sequence $\{z_n\}_{n \geqslant 1} \subset \D$, repeated according to multiplicity, 
		is the zero set of a nonconstant $H^2$ function if and only if it satisfies the Blaschke condition \eqref{eq:BlaschkeCondition}.
	\end{Theorem}

	Other examples of inner functions are furnished by the following construction.
	For a positive, finite, singular (with respect to $m$) Borel measure $\mu$ on $\T$, we claim that the analytic function
	\begin{equation}\label{eq:SingularInnerFunction}
		\qquad S_{\mu}(z) := \exp\left(-\int \frac{\zeta + z}{\zeta - z} d \mu(\zeta)\right), \quad z \in \D,
	\end{equation}
	is inner. Such a function is known as a \emph{singular inner function}.  First notice that for any $z \in \D$, 
	\begin{align*}
		|S_{\mu}(z)| 
		& = \exp\left[ \Re\left(-\int \frac{\zeta + z}{\zeta - z} d \mu(\zeta)\right)\right]\\
		& = \exp\left(-\int_{\T} \frac{1 - |z|^2}{|\zeta - z|^2} d \mu(\zeta)\right)\\
		& \leqslant 1
	\end{align*}
	since both the measure $\mu$ and the \emph{Poisson kernel}
	\begin{equation*}
		P_{z}(\zeta) := \frac{1 - |z|^2}{|\zeta - z|^2}, \tag{$\zeta \in \T, z \in \D$}
	\end{equation*} 
	are nonnegative.  Since $S_{\mu}$ belongs to $H^{\infty}$, it follows that
	$S_{\mu}$ has nontangential boundary values $m$-a.e.~on $\T$.  To show that these
	boundary values are almost everywhere unimodular, we require some basic facts from the theory of harmonic functions.  Let
	\begin{equation*}
		(D \mu)(w) := \lim_{t \to 0^{+}} \frac{\mu\big((e^{-i t} w, e^{i t} w)\big)}{2 t}  \tag{$w \in \T$}
	\end{equation*}
	denote the \emph{symmetric derivative} of $\mu$ on $\T$,
	where $(e^{-i t} w, e^{i t} w)$ denotes the circular arc subtended by the points $e^{-it} w$ and $e^{i t} w$.
	We also have the identity 
	\begin{equation}\label{Poisson-sym}
		\lim_{r \to 1^{-}} \int_{\T} P_{r w}(\zeta) d\mu(\zeta) = (D \mu)(w). \qquad \mbox{$m$-a.e. $w \in \T$}.
	\end{equation}
	Since $\mu$ is singular, it follows that $D \mu = 0$ holds $m$-almost everywhere and so
	it now follows from the identity 
	\begin{equation*}
		|S_{\mu}(z)| = \exp\left(-\int_{\T} P_{z}(\zeta) d\mu(\zeta)\right)
	\end{equation*}
	that $S$ has unimodular boundary values $m$-almost everywhere.  For instance, if $\mu = \delta_1$ denotes the point mass at $\zeta = 1$, then
	\begin{equation*}
		S_{\delta_1}(z) = \exp\left(\frac{z + 1}{z - 1}\right).
	\end{equation*}
	This type of inner function is often called an {\em atomic inner function}. 
	
	\begin{Theorem}
		Every inner function $u$ can be factored uniquely as
		\begin{equation*}
			u = e^{i \gamma} B_{\Lambda} S_{\mu},
		\end{equation*}
		where $\gamma \in [0,2\pi)$, $\Lambda$ is a Blaschke sequence, and $\mu$ is a positive singular measure on $\T$. 
		Conversely, any such product is inner. 
	\end{Theorem}		
	
	The reader is warned that, as an abuse of language, one frequently permits
	a function of the form $e^{i \gamma} B_{\Lambda}$ (resp.~$e^{i \gamma} S_{\mu}$) to be called
	a Blaschke product (resp.~a singular inner function).  For the sake of convenience, we adopt this
	common practice here in these notes.
	
\begin{Definition}\label{divides-def}
Let $u_1, u_2$ be inner functions and $f \in H^2$. 
\begin{enumerate}\addtolength{\itemsep}{0.25\baselineskip}
\item We say that $u_1$ {\em divides} $u_2$, written $u_1|u_2$, if $u_2/u_1 \in H^{\infty}$. 
\item We say that $u_1$ {\em divides} $f$, written $u_1|f$, if $f/u_1 \in H^2$. 
\item We say that $u_1$ and $u_2$ are {\em relatively prime} if the only inner divisors of both $u_1$ and $u_2$ are constant functions of unit modulus.
\end{enumerate}
\end{Definition}

For example, if $u_1$ and $u_2$ are Blaschke products with simple zeros, then $u_1$ divides $u_2$ if and only if  the zero set of $u_1$ is contained in the zero set
of $u_2$. Moreover, $u_1$ is relatively prime to $u_2$ if and only if $u_1$ and $u_2$ have no common zeros.

	\begin{Definition}
		An \emph{outer function} is an analytic function $F$ on $\D$ of the form 
		\begin{equation}\label{eq:Outer}
			F(z) = e^{i\gamma} \exp \left( \int_{\T}{\frac{\zeta+z}{\zeta-z} \phi(\zeta)\,dm(\zeta)} \right),
		\end{equation}
		where $\gamma$ is a real constant and $\phi$ is a real-valued function in $L^1$.
	\end{Definition}
	
	  The significance of the somewhat
	unwieldy formula in \eqref{eq:Outer} lies in the fact that 
	\begin{equation*}
		\log|F(z)|
		= \int_{\T} P_{z}(\zeta) \phi(\zeta) dm(\zeta),
	\end{equation*}
	from which
	it follows that $\log|F(z)|$ equals the harmonic extension to $\D$ of
	the boundary function $\phi:\T\to\R$ so that $\phi = \log|F|$ a.e.~on $\T$.
	In particular, an outer function is determined up to a unimodular constant factor by its modulus on $\T$.
	On the other hand, it is possible to show that every $f$ in $H^2$ that does not vanish identically satisfies
	\begin{equation*}
		\int_{\T} \log |f(\zeta)| \,dm(\zeta) > -\infty
	\end{equation*}
	and hence one can form the outer function \eqref{eq:Outer} corresponding to the boundary data $\phi = \log|f|$ \cite[Thm.~2.2]{Duren}.
	Putting this all together, we can now state the canonical factorization for $H^2$ functions \cite[Thm.~2.8]{Duren}.
	
	\begin{Theorem}
		Every function $f$ in $H^2 \backslash \{0\}$ has a unique factorization of the form
		\begin{equation}\label{eq:BSF}
			f = BSF
		\end{equation}
		where $B$ is a Blaschke product,
		$S$ is a singular inner function, and $F$ is an outer function in $H^2$.   
		Conversely, any product of the form \eqref{eq:BSF} belongs to $H^2$.
	\end{Theorem}
	
\subsection{Bounded type}

In order to adequately discuss the cyclic vectors for the backward shift operator, we need two more additional 
classes of meromorphic functions on $\D$. 

\begin{Definition}\label{bounded-type}
 Let $f$ be a meromorphic function on $\D$. 
\begin{enumerate}
\item We say $f$ is of {\em bounded type} if $f$ can be written as quotient of two bounded analytic functions. The set of functions of bounded type is denoted by $N$ (often called the Nevanlinna class). 
\item We say $f$ is in the {\em Smirnov class} if is the quotient of two bounded analytic functions where the denominator is an outer function. The Smirnov class is denoted by $N^{+}$. Notice in this case that $N^{+}$ is a space of analytic functions on $\D$ since the denominator is outer and outer functions have no zeros on $\D$. 
\end{enumerate}
\end{Definition}

\begin{Remark}\label{N-facts}
It is known that $H^2(\D) \subset N^{+}$ and that every $f \in N$ has finite non-tangential limits almost everywhere \cite[Ch.~2]{Duren}. Furthermore, if $f \in N^{+}$ and the boundary function belongs to $L^2$, then $f \in H^2$. This fact is no longer true for $f \in N$\footnote{Even if $f$ is analytic on $\D$, this is no longer true. Just consider the function $f(z) = \exp(-\frac{z + 1}{z - 1})$, which is the reciprocal of the atomic inner function defined earlier. This function belongs to $N$, is analytic on $\D$, has unimodular boundary values, but does not belong to $H^2$ since some integral estimates will show that $f$ does not have bounded integral means.}
\end{Remark}
	
\subsection{Beurling's Theorem}	

	Of supreme importance in the world of operator-related function theory are the shift operators.
	Chief among these is the \emph{unilateral shift} $S:H^2 \to H^2$ defined by
	\begin{equation*}
		[Sf](z) = zf(z),
	\end{equation*}
	or, in terms of Taylor coefficients, by
	$$S(a_0,a_1,\ldots) = (0,a_0,a_1,\ldots).$$ Because of its ubiquity in the realm of operator theory,
	one often refers to $S$ as \emph{the} shift operator. One easily sees that $S$ is an isometry that is not unitary. 
	The adjoint $S^{*}$ of the unilateral shift is the \emph{backward shift} $S^*:H^2\to H^2$ given by
	\begin{equation*}
		[S^*f](z) = \frac{f(z)-f(0)}{z},
	\end{equation*}
	or, in terms of Taylor coefficients, by
	$$S^*(a_0,a_1,\ldots) = (a_1,a_2,\ldots).$$
		
	If $u$ is inner, then the operator $f \mapsto u f$ is an isometry on $H^2$ and thus $u H^2$ is a subspace 
	(i.e., a closed linear manifold) of $H^2$.  Moreover, assuming that $u$ is a nonconstant inner function,  that $u H^2$ is a nontrivial invariant subspace for the operator $S$. 
	A celebrated theorem of Beurling \cite{B} (also see any of the standard texts 
	mentioned above for a complete proof and further discussions and generalizations) says that these are all of them. 

	\begin{Theorem}[Beurling's Theorem]
		The nontrivial invariant subspaces of $H^2$ for the unilateral shift $S$ are precisely the subspaces
		$$u H^2 := \{ u h : h\in H^2\},$$ where $u$ is an inner function.		
		Moreover, $f$ is cyclic for $S$, i.e., $$\bigvee\{ f,Sf,S^2f,\ldots\} = \bigvee \{ qf: \text{$q$ is a polynomial}\} = H^2,$$
		if and only if $f$ is an outer function.
	\end{Theorem}
		
	From an operator theoretic perspective, Beurling's Theorem is notable for providing an explicit description of the lattice  
	of invariant subspaces for the unilateral shift operator.  Indeed, 
	\begin{equation}\label{eq:u1h2u2}
		u_1H^2 \subseteq u_2 H^2 \iff u_1/u_2 \in H^{\infty}.
	\end{equation}
	For our purposes, however, it is the invariant subspaces  for the backward shift
	that are of greatest importance.  These are the so-called \emph{model spaces} that are the primary focus of our investigations.

\section{Model spaces}
	The details of much of the following material can be found in the original sources, which we attempt to quote whenever possible, 
	as well as the texts \cite{CR, CMR, N1, N2,  RR, RS}.  On the other hand, many of the results discussed below are part of the
	folklore on the subject and occasionally proper references cannot be readily identified.

\subsection{Basic properties}\label{SubsectionBasic}
	We are now ready to introduce our primary object of study.

	\begin{Definition}
		If $u$ is an inner function, then the corresponding \emph{model space} is
		\begin{equation}\label{eq:Ku2}
			\K_u := (uH^2)^{\perp} = H^2 \ominus u H^2.
		\end{equation}
	\end{Definition}

	The definition above is somewhat unenlightening since the function theoretic properties of the model space $\K_u$ 
	do not immediately present themselves.  A more direct description of $\K_{u}$ via the boundary values of these functions comes from the following result. 

	\begin{Proposition}\label{Prop:fgzu}
		For inner $u$, the model space $\K_{u}$ is the set of functions $f$ in $H^2$ such that 
		$f = \overline{gz}u$ almost everywhere on $\T$ for some $g$ in $H^2$.  In other words,
		\begin{equation*}
			\K_{u} = H^2 \cap u \overline{z H^2},
		\end{equation*}
		where the right hand side is regarded as a set of functions on $\T$.
	\end{Proposition}
	
	\begin{proof} 
		For each $f$ in $H^2$, we see that
		\begin{equation*}
			\inner{f,uh} = 0,\, \forall h \in H^2 \quad\iff\quad \inner{\overline{u}f,h} = 0,\, 
			\forall h \in H^2
			\quad\iff\quad\overline{u}f \in \overline{zH^2}.
		\end{equation*}
		Since $u\overline{u}=1$ almost everywhere on $\T$, we see that an $H^2$ function $f$ belongs to the orthocomplement of $uH^2$ if and only if
		$f$ belongs to $u \overline{zH^2}$ (i.e., $f = \overline{gz}u$ for some $g$ in $H^2$).\footnote{In fact, 
		$g$ actually belongs to $\K_u$ as well (see Section \ref{SectionConjugation}).}
	\end{proof}
	
	Just as the Beurling-type subspaces $uH^2$ constitute the nontrivial invariant subspaces for
	the unilateral shift on $H^2$, the subspaces $\K_u$ play an analogous role for
	the backward shift.

	\begin{Corollary}\label{cyclic-vec-S-star}
		The model spaces $\K_u$, where $u$ is inner, are precisely the proper invariant subspaces of $H^2$ for the backward shift
		\begin{equation}\label{eq:BSD}
			f(z) \mapsto \frac{f(z) - f(0)}{z} ,\qquad\qquad
			(a_0,a_1,a_2,\ldots) \mapsto (a_1,a_2,a_3,\ldots).
		\end{equation}
	\end{Corollary}
	
	The following corollary is an immediate consequence of \eqref{eq:u1h2u2}
	(recall what it means for one inner function to {\em divide} another from Definition \ref{divides-def}). 

	\begin{Corollary}
		If $u_1,u_2$ are inner functions, then
		$$u_1 | u_2 \iff \K_{u_1} \subseteq \K_{u_2}.$$
	\end{Corollary}
	
	\begin{Remark}
		We pause here to mention the mildly surprising fact, not widely known, that if $u$ is not a finite Blaschke product, then $\K_u$ 
		contains a linearly ordered chain of $S^*$-invariant subspaces of uncountable length.  This is obvious if $u$ has 
		a singular inner factor $S_{\mu}$ since $\K_{S_{\alpha \mu}} \subseteq \K_u$ for $0 \leqslant \alpha \leqslant 1$.  On the other hand,
		if $u$ is an infinite Blaschke product, then things are not so clear.
		Suppose that $u = \prod_{n=1}^{\infty} b_n$ where each $b_n$ is a Blaschke factor.  Let $I \subseteq \R$ be a nonempty interval and
		let $\tau: \mathbb{Q} \cap I \rightarrow \mathbb{N}$ be a bijection.  For each $\alpha$ in $I$, define
		$I_{\alpha} = \{ q \in \mathbb{Q} \cap I : q < \alpha \}$ and notice that each $\alpha$ corresponds to a distinct Blaschke product
		$u_{\alpha} = \prod_{n \in \tau (I_{\alpha}) }  b_n $ that divides $u$.  This yields the desired chain of $S^*$-invariant subspaces.
	\end{Remark}

	Building upon the proof of Proposition \ref{Prop:fgzu}, we have the following important result
	\cite[Thm.~3.1.5]{DSS}. 
	
	\begin{Proposition}\label{Prop:Generator}
		A function $f$ in $H^2$ is noncyclic for $S^*$ if and only if there exists a function $g$ in $H^2$
		and an inner function $u$ such that 
		\begin{equation}\label{eq:fugu}
				f = \overline{gz}u
		\end{equation}	
		almost everywhere on $\T$.
		If $u$ and the inner factor of $g$ are relatively prime, then 
		the $S^*$-invariant subspace of $H^2$ generated by $f$ is $\K_u$ itself:
		\begin{equation*}
			\bigvee\{f, S^*f, S^{*2}f,\ldots\} = \K_u.
		\end{equation*}
	\end{Proposition}
	
\begin{Remark}
We will see another description of $\K_u$ as well as the noncyclic vectors for $S^{*}$ when we discuss pseudocontinuations in Section \ref{Section:BB}. 
\end{Remark}

	Being invariant under the backward shift operator, it is not surprising that the spaces
	$\K_u$ are also invariant under certain functions of the backward shift.  
	In what follows, we let $P:L^2 \to H^2$ denote the \emph{Riesz projection}
	\begin{equation}\label{eq:RieszProjection}
		P(\ldots,a_{-1},a_0,a_1,a_2,\ldots) = (a_0,a_1,a_2,\ldots),
	\end{equation}
	the orthogonal projection which returns the ``analytic part'' of a Fourier series in $L^2$.
	For instance, $P(1+ 2 \cos \theta) = P(  e^{-i\theta} + 1 + e^{i\theta}) = 1+ e^{i\theta} = 1+\zeta$.
	We also remark that, as an orthogonal projection, the operator $P$ is self-adjoint and hence
	satisfies $\inner{Pf,g} = \inner{f,Pg}$ for all $f,g$ in $L^2$.

	\begin{Definition}\label{DefinitionToeplitz}
		For $\phi$ in $L^{\infty}(\T)$ the \emph{Toeplitz operator} $T_{\phi}:H^2\to H^2$ with \emph{symbol} $\phi$
		is defined by
		\begin{equation*}
			T_{\phi}(f) = P(\phi f).
		\end{equation*}
	\end{Definition}	
	
	The study of Toeplitz operators is a vast subject and we refer the reader to the texts \cite{MR0361893, MR2270722} 
	for the basics and to \cite{MR2223704} for an encyclopedia on the subject. 
	Note that when $\phi$ belongs to $H^{\infty}$, the Toeplitz operator $T_{\phi} f = \phi f$ is just a multiplication operator. 
	A simple calculation shows that $T_z$ and $T_{\overline{z}}$ are precisely the forward 
	and backward shift operators on $H^2$.  The operator $T_z:H^2\to H^2$ enjoys an 
	$H^{\infty}$-functional calculus
	given by $\phi(T_z) = T_{\phi}$ for $\phi$ in $H^{\infty}$ \cite{MR2760647}.  Being
	$T_{\overline{z}}$-invariant already, it should come as little surprise that the spaces $\K_u$ are also invariant under conjugate-analytic
	Toeplitz operators (i.e., Toeplitz operators of the form $T_{\overline{\phi}}$ where $\phi$ belongs to $H^{\infty}$). 

	\begin{Proposition}\label{Prop:CATO}
		If $\phi \in H^{\infty}$ and $u$ is inner then $T_{\overline{\phi}}\K_u \subseteq \K_u$.
	\end{Proposition}

	\begin{proof}
		For each $f$ in $\K_u$ we have
		\begin{equation*}
			\inner{ T_{\overline{\phi}}f, uh} 
			 = \inner{ P( \overline{\phi} f),uh} 
			 = \inner{ \overline{\phi} f , P(uh)} 
			 = \inner{ \overline{\phi}f,uh}
			 = \inner{f, u(\phi h)} 
			 = 0. \qedhere
		\end{equation*}
	\end{proof}

	Proposition \ref{Prop:CATO} shows that the spaces
	$\K_u$ enjoy the so-called \emph{$F$-property}.\footnote{In general, a set
		$\mathscr{C}$ of functions contained in $H^2$ has the \emph{$F$-property} 
		if whenever $\theta$ divides $f$, then $f/\theta \in \mathscr{C}$. 
		Good sources for this are \cite{MR505686, MR643387}.}

	\begin{Proposition}
		If $f$ belongs to $\K_u$ and $\theta$ is an inner function that divides $f$, i.e., $f/\theta$ belongs to $H^2$, 
		then $f/\theta$ also belongs to $\K_u$.  In particular, the outer factor of any function in $\K_u$ also
		belongs to $\K_u$.
	\end{Proposition}
	
	\begin{proof}
		Simply observe that $T_{\overline{\theta}}f = P(\overline{\theta}f) = P(f/\theta) = f/\theta$ since $f/\theta$
		is in $H^2$.  By the preceding proposition, $f/\theta$ belongs to $\K_u$.
	\end{proof}

\subsection{Finite dimensional model spaces}
	The simplest examples of model spaces
	are those corresponding to finite Blaschke products
	\begin{equation*}
		u(z) = \prod_{j = 1}^{n} \frac{z - \lambda_j}{1 - \overline{\lambda_j} z}.  \tag{$\lambda_j \in \D$}
	\end{equation*}
	Indeed, these are the only model spaces that are finite-dimensional and whose elements can be
	completely characterized in an explicit fashion.
	To do this, one should first notice that the Cauchy-Szeg\H{o} kernel $c_{\lambda}$ (from \eqref{szego-k})
	belongs to $\K_u$ whenever $\lambda$ is a zero of $u$.  Indeed, if $u(\lambda) = 0$,
	then clearly $\inner{uh,c_{\lambda}} = u(\lambda)h(\lambda) = 0$ for all $h$ in $H^2$.

	\begin{Proposition}\label{Prop:FDMS}
		If $u$ is a finite Blaschke product with zeros $\lambda_1,\lambda_2,\ldots,\lambda_n$, repeated
		according to their multiplicity, then
		\begin{equation}\label{eq:FDMSP}
			\K_u
			= \left\{\, \frac{a_0 + a_1 z + \cdots + a_{n-1} z^{n-1}}{(1-\overline{\lambda_1}z)(1-\overline{\lambda_2}z)\cdots (1-\overline{\lambda_n}z)} :
			a_0,a_1,\ldots,a_{n-1} \in \C\right\}.
		\end{equation}
		In particular, $\K_{z^n}$ is the space of all polynomials of degree $\leqslant n - 1$.
	\end{Proposition}

	\begin{proof}
		We provide the proof in the special case where the zeros $\lambda_i$ are distinct.
		Since $\inner{u h,c_{\lambda_i}} = u(\lambda_i)h(\lambda_i)=0$ for all $h$ in $H^2$,  it follows that
		\begin{equation*}
			\operatorname{span} \{c_{\lambda_1}, c_{\lambda_2}, \ldots, c_{\lambda_n}\} \subseteq \K_{u}. 
		\end{equation*}
		If $f(\lambda_i) = \inner{f,c_{\lambda_i}} = 0$ for all $i$, then $u|f$ and hence $f$
		belongs to $u H^2$.  Thus
		\begin{equation*}
			\operatorname{span} \{c_{\lambda_1}, c_{\lambda_2}, \ldots, c_{\lambda_n}\}^{\perp} \subseteq \K_{u}^{\perp} . 
		\end{equation*}
		Since $\K_{u} = \operatorname{span} \{c_{\lambda_1}, c_{\lambda_2}, \ldots, c_{\lambda_n}\}$ 
		the result follows by simple algebra.  To be more specific, any linear combination of the Cauchy kernels
		$\{c_{\lambda_j}: 1 \leqslant j \leqslant n\}$ can be expressed as a rational function
		of the type prescribed in \eqref{eq:FDMSP}.  Conversely, any expression of the type encountered
		in \eqref{eq:FDMSP} can be decomposed, via partial fractions, into a linear combination
		of the functions $c_{\lambda_1}, c_{\lambda_2}, \ldots, c_{\lambda_n}$.
	\end{proof}
	
	If $\lambda$ has multiplicity $m$ as a zero of $u$, then one must also include the functions
	\begin{equation*}
		c_{\lambda}, c_{\lambda}', c_{\lambda}'',\ldots, c_{\lambda}^{(m-1)}
	\end{equation*}
	in place of  $c_{\lambda}$
	in the preceding proof.  Along similar lines, the proof of Proposition \ref{Prop:FDMS} and the preceding
	comment provides us with the following useful fact.
	
	\begin{Proposition}
		Suppose that $u$ is the finite Blaschke product with distinct zeros $\lambda_1,\lambda_2,\ldots,\lambda_n$
		with  respective multiplicities $m_1,m_2,\ldots,m_n$, then
		\begin{equation*}
			\K_u = \operatorname{span} \{ c_{\lambda_i}^{(\ell_i-1)} : 1 \leqslant i \leqslant n, \,1 \leqslant \ell_i \leqslant m_i \}.
		\end{equation*}
	\end{Proposition}
	
	In fact, the preceding observation makes it clear why $\dim \K_u < \infty$ occurs if and only if
	$u$ is a finite Blaschke product.  If $u$ has a factor that is an infinite Blaschke product, then Proposition \ref{Prop:CKLI}
	ensures that $\K_u$ contains an infinite, linearly independent set (namely the Cauchy kernels corresponding
	to the distinct zeros of $u$).  On the other hand, 
	if $u$ is a singular inner function, then $u^{1/n}$ is an inner function that divides $u$ whence
	$\K_{u^{1/n}} \subseteq \K_u$ for $n \geqslant 1$.   We may find an infinite orthonormal sequence in $\K_u$ by selecting unit vectors
	$f_n$ in $\K_{u^{1/n}} \ominus \K_{u^{1/(n+1)}}$, from which it follows that $\dim \K_u = \infty$.
	For the general case we point out the following decomposition that is interesting in its 
	own right.    
	
	\begin{Proposition} 
		If $\{u_j\}_{j \geqslant 1}$ is a possibly finite sequence of inner function such that 
		$u = \prod_{j \geqslant 1} u_j$ exists, then
		\begin{equation*}
			\K_u = \K_{u_1} \oplus \bigoplus_{n \geqslant 2} \Big(\prod_{j = 1}^{n - 1} u_j \Big) \K_{u_n}.
		\end{equation*}
		In particular, if $u$ and $v$ are inner functions, then 
		\begin{equation}\label{eq:Kuv}
			\K_{u v} = \K_u \oplus u \K_{v}.
		\end{equation}
	\end{Proposition}

	We refer the reader to \cite{AC70a, Kriete} for further details, although we provide a proof
	of the special case \eqref{eq:Kuv} in Subsection \ref{Subsection:RKBP}.

\subsection{Three unitary operators} The following three transformations from one model space to another are often useful. 

	\begin{Proposition}\label{three-unitary}
	Suppose $u$ is a fixed inner function. 
	\begin{enumerate}
		\item If $w \in \D$, then
		\begin{equation*}
			 f \mapsto \frac{\sqrt{1 - |w|^2}}{1 - \overline{w} u} f
		\end{equation*}
		defines a unitary operator from $\mathcal{K}_{u}$ onto $\mathcal{K}_{\frac{u-w}{1-\overline{w}u}}$.
		\item If $\phi$ is a disk automorphism, i.e., $\phi(z) = \zeta (z - a)(1 - \overline{a} z)^{-1}$ for some $a \in \D$ and $\zeta \in \T$, then 
		\begin{equation*}
		f \mapsto \sqrt{\phi'} (f \circ \phi)
		\end{equation*}
		defines a unitary operator from $\K_{u}$ onto $\K_{u \circ \phi}$.
		\item If $u^{\#}(z) := \overline{u(\overline{z})}$, then, in terms of boundary functions, the map
		\begin{equation*}
		f(\zeta) \mapsto \overline{\zeta} f(\overline{\zeta}) u^{\#}(\zeta)
		\end{equation*}
		is a unitary operator from $\K_u$ into $\K_{u^{\#}}$.
	
	\end{enumerate}
	\end{Proposition}

	The first unitary operator is due to Crofoot \cite{Crofoot} (a more detailed discussion of these so-called \emph{Crofoot transforms}
	can be found in \cite[Sec.~13]{Sarason}).  When dealing with a model space $\K_u$ where $u$ has a nontrivial Blaschke factor,
	one can often make the simplifying assumption that $u(0) = 0$.  Also of great importance is the fact that Crofoot transforms
	intertwine the conjugations (see Section \ref{SectionConjugation}) on the corresponding models spaces \cite[Lem.~3.1]{Sarason}. The second unitary operator is clearly unitary on $H^2$ (change of variables formula). Showing that its restriction to $\K_u$ has the correct range is a little tricky and the proof of this can be found in \cite[Prop.~4.1]{TTOSIUES}. The third unitary operator depends on a discussion on conjugations that we take up in more detail in Section \ref{SectionConjugation}. Indeed, on the face of it, the map does not even seem to take analytic functions to analytic functions. However, when one thinks of model space functions in terms of their boundary values as in Proposition \ref{Prop:fgzu}, everything works out. The proof can be found in \cite[Lemma 4.3]{TTOSIUES}.

\section{Model operators}
	One of the main reasons that model spaces are worthy of study in their own right
	stems from the so-called \emph{model theory} developed by Sz.-Nagy and Foia\c{s}, which
	shows that a wide range of Hilbert space operators can be realized concretely as restrictions of the backward shift
	operator to model spaces.  These ideas have since been generalized in many directions 
	(e.g., de Branges-Rovnyak spaces, vector-valued Hardy spaces, etc.) and we make no attempt to provide
	an encyclopedic account of the subject, referring the read instead to the influential texts \cite{Bercovici, MR2760647, N1, RR, N3}.

\subsection{Contractions}
	In the following, we let $\h$ denote a separable complex Hilbert space.
	If $T$ is an arbitrary bounded operator on $\h$, then we may assume that
	$\norm{T} \leqslant 1$ (i.e., $T$ is a \emph{contraction}).  As such, $T$ enjoys a decomposition
	of the form $T = K \oplus U$ (see \cite[p.~8]{MR2760647} for more details) where $U$ is a unitary operator and $K$ is a \emph{completely nonunitary (CNU) contraction} 
	(i.e., there does not exist a reducing subspace for $K$ upon which $K$ is unitary).
		
	Since the structure 
	and behavior of unitary operators is well-understood, via the spectral theorem, the study of arbitrary bounded Hilbert space operators can
	therefore be focused on CNU contractions.  With a few additional hypotheses, one can obtain a concrete
	\emph{functional model} for such operators.
	In light of the fact that the following theorem justifies, to an extent, 
	the further study of model subspaces, we feel obliged to
	provide a complete proof.  Moreover, the proof itself is surprisingly simple 
	and is worthy of admiration for its own sake.

	\begin{Theorem}[Sz.-Nagy-Foia\c{s}]\label{NF-Thm}
		If $T$ is a contraction on a Hilbert space that satisfies
		\begin{enumerate}
			\addtolength{\itemsep}{0.5\baselineskip}
			\item $\norm{T^n \vec{x}}\to 0$ for all $\vec{x}\in \h$, 
			\item $\rank(I - T^*T) = \rank(I -TT^*) = 1$,
		\end{enumerate}
		then there exists an inner function $u $ such that $T$ is unitarily equivalent to 
		$S^*|\K_{u }$, where $S^{*}$ is the backward shift operator on $H^2$. 
	\end{Theorem}
			
	\begin{proof}
		We let $\cong$ denote the unitary equivalence of Hilbert spaces or their operators.
		Since the \emph{defect operator} $D = \sqrt{I - T^*T}$ has rank $1$, we see that $\ran D \cong \C$ so that
		\begin{equation}\label{Hwidehat}
			\widetilde{\h} := \bigoplus_{n=1}^{\infty} \ran D \cong H^2.
		\end{equation}
		It follows that for each $n \in \N$ we have 
		\begin{align*}
			\sum_{j=0}^n \norm{DT^j \vec{x}}^2
			&= \sum_{j=0}^n \big< (I - T^*T)^{\frac{1}{2}} T^j\vec{x}, (I - T^*T)^{\frac{1}{2}} T^j\vec{x} \big>\\
			&= \sum_{j=0}^n \big< (I - T^*T)T^j\vec{x}, T^j\vec{x} \big>\\
			&= \sum_{j=0}^n \Big(\inner{T^j \vec{x}, T^j \vec{x}} - \inner{T^*T T^j\vec{x}, T^j\vec{x}} \Big) \\
			&= \sum_{j=0}^n \big(\norm{T^j\vec{x}}^2 - \norm{T^{j+1}\vec{x}}^2 \big)\\
			&= \norm{\vec{x}}^2 - \norm{T^{n+1}\vec{x}}^2.
		\end{align*}
		Since, by hypothesis,  $\|T^n \vec{x}\| \to 0$ for each $\vec{x} \in \mathcal{H}$, we conclude that
		\begin{equation*}
		    \sum_{j=0}^{\infty} \norm{DT^j \vec{x}}^2 = \norm{\vec{x}}^2
		    \tag{$\vec{x} \in \h$}
		\end{equation*}
		  and hence the operator $\Phi:\h\to H^2$ defined by
		\begin{equation*}
			\Phi \vec{x} = (D\vec{x}, DT\vec{x}, DT^2\vec{x}, DT^3 \vec{x},\ldots)
		\end{equation*}
		is an isometric embedding of $\h$ into $H^2$ (here we have identified a function in $H^2$
		with its sequence of Taylor coefficients).
		Since $\Phi$ is an isometry, its image
	 			\begin{equation*}
			\Phi\h = (D\h, DT\h, DT^2\h,\ldots)
		\end{equation*}
		is closed in $H^2$ and clearly $S^*$-invariant.  
		Therefore, by Corollary \ref{cyclic-vec-S-star},  $\ran \Phi = \K_u$ for some $u$ (the possibility that $\ran \Phi = H^2$ is ruled out
		because the following argument would show that $T \cong S^*$, violating the assumption that
		$\rank(I - TT^*) = 1$).		
		Now observe that	
		\begin{equation*}
			\Phi T\vec{x}
			= (DT\vec{x}, DT^2\vec{x}, DT^3\vec{x},\ldots)
			= S^* \Phi \vec{x}. \tag{$\vec{x} \in \h$}
		\end{equation*}
		Letting $U:\h\to\ran\K_u$ denote the unitary operator obtained from $\Phi$ by reducing its codomain from
		$H^2$ to $\K_u$,  it follows that
		\begin{equation*}
			UT = (S^*|\K_u)U.
		\end{equation*}
		Thus $T$ is unitarily equivalent to the restriction of $S^*$ to $\K_u$.
	\end{proof}
	
	The case of higher defect indices (i.e., $\rank(I - T^*T) = \rank(I -TT^*) \geqslant n$, $n > 1$),  is treated by moving to vector-valued Hardy spaces $H^2(\mathcal{E})$ and
	operator-valued inner functions.  However, in making such a 
	move one sacrifices a large variety of tools and techniques inherited from classical function theory
	(making the theory of model spaces more difficult and less interesting from our perspective).
	For instance, the multiplication of operator-valued inner functions is no longer commutative
	and the corresponding factorization theory is more complicated.

\subsection{Spectrum of an inner function}
	In light of the Sz.-Nagy-Foia\c{s} Theorem, the restriction of the backward shift to the spaces $\K_u$
	is of premier importance.  However, it turns out that the \emph{compressed shift}, the compression $S_u:\K_u\to\K_u$
	of the unilateral shift to $\K_u$, is more prevalent in the literature.  Here 
	\begin{equation*}
		S_uf = P_u(zf),
	\end{equation*}
	where $P_u$ denotes the orthogonal
	projection from $L^2$ onto $\K_u$ (see Subsection \ref{Subsection:RKBP}).
	In reality, the distinction alluded to above is artificial since the operator $S_u$ is unitarily equivalent to the restriction of the backward
	shift to the space $\K_{u^{\#}}$ where $u^{\#}(z) = \overline{u(\overline{z})}$. In fact the unitary operator which intertwines $S_{u}$ and $S^{*}|\K_{u^{\#}}$ is the one given in Proposition \ref{three-unitary}. See \cite[Lemma 4.5]{TTOSIUES} for more details on this. 
	
	There happens to be two convenient ways to describe the spectrum $\sigma(S_u)$ of $S_u$.
	In fact, one can show that $\sigma(S_u)$
	coincides with the so-called \emph{spectrum} of the inner function $u$, as defined below
	(this result is sometimes referred to as the Liv\v{s}ic-M\"oller Theorem).

	\begin{Definition}
		The \emph{spectrum} of an inner function is the set 
		\begin{equation}\label{eq:SpecDef}
			\sigma(u) := \left\{\lambda \in \D^{-}: \liminf_{z \to \lambda} |u(z)| = 0\right\}.
		\end{equation}
	\end{Definition}
	
	It follows from the preceding definition, for instance, that every zero of $u$ in $\D$ belongs to $\sigma(u)$.
	Moreover, any limit point of zeros of $u$ must also lie in $\sigma(u)$.  On the other hand, if $u$ is the singular inner function associated to
	the singular measure $\mu$,
	then the nontangential limit of $u$ is zero $\mu$-almost everywhere \cite[Thm.~6.2]{Garnett}.   Even further, 
	if $\lambda \in \D$ and $u(\lambda) = 0$ then $c_{\lambda}(z) = (1 - \overline{\lambda} z)^{-1}$ belongs to $\K_u$ and satisfies
	$S^{*} c_{\lambda} = \overline{\lambda} c_{\lambda}$.  In other words, $\overline{\lambda}$ belongs to $\sigma_{p}(S^{*}|\K_u)$, the point spectrum (i.e., set of eigenvalues) of $S^{*}|\K_u$. 
	These observations suggest the following important theorem that
	provides us with a convenient description
	of the spectrum of an inner function in terms of its canonical factorization.
		
	\begin{Theorem}\label{TheoremSpectrum}
		If $u = B_{\Lambda} S_{\mu}$, where $B_{\Lambda}$ is a Blaschke product with zero sequence
		$\Lambda$ and $S_{\mu}$ is a singular inner function with corresponding
		singular measure $\mu$, then
		\begin{enumerate}\addtolength{\itemsep}{0.5\baselineskip}
			\item $\sigma(S_u) =  \sigma(u) = \Lambda^- \cup \operatorname{supp} \mu$,
			\item $\sigma_{\text{p}}(S^{*}|\K_u) = \{\overline{\lambda}: \lambda \in \Lambda\}$, 
			\item $\sigma_{\text{p}}(S_u) = \Lambda$,
		\end{enumerate}
		where $\Lambda^-$ denotes the closure of $\Lambda$.
	\end{Theorem}		

\section{Reproducing kernels}\label{Section:RK}

	As a closed subspace of the reproducing kernel Hilbert space\footnote{A Hilbert space $\mathcal{H}$ of analytic functions on a domain $\Omega \subset \C$ is called a {\em reproducing kernel Hilbert space} if for each $\lambda \in \Omega$ there is a function $K_{\lambda} \in \mathcal{H}$ for which $f(\lambda) = \langle f, K_{\lambda}\rangle_{\mathcal{H}}$ for all $f \in \mathcal{H}$. All of the classical Hilbert spaces of analytic functions on the disk (e.g., Hardy, Bergman, Dirichlet, etc.) are reproducing kernel Hilbert spaces and understanding the properties of the kernels often yields valuable information about the functions in $\mathcal{H}$. A few good sources for this are \cite{MR0051437, Paulsen, AMC}.} $H^2$, each model space
	$\K_u$ itself possesses a reproducing kernel.  In this section, we identify these kernels and explore
	their basic properties.

\subsection{Basic properties}\label{Subsection:RKBP}
	Recalling that the Cauchy kernels $c_{\lambda} = (1- \overline{\lambda}z)^{-1}$ are the 
	reproducing kernels for the Hardy space (see Subsection \ref{Subsection:CSZK}), 
	let us attempt to compute the corresponding reproducing kernels for
	$\K_u$.  We first observe that if $f = uh$ belongs to $u H^2$, then
	\begin{equation*}
		f(\lambda) = u(\lambda)h(\lambda) = u(\lambda) \inner{h,c_{\lambda}} = 
		u(\lambda) \inner{f\overline{u},c_{\lambda}} = \inner{f, \overline{u(\lambda)}u c_{\lambda}},
	\end{equation*}
	from which it follows that the reproducing kernel for $uH^2$ is given by
	\begin{equation*}	
		\overline{u(\lambda)}u(z) c_{\lambda}(z).
	\end{equation*}
	If $f$ belongs to $\K_u$, then it follows that
	\begin{align*}
		f(\lambda) 
		&= \inner{f,c_{\lambda}} \\
		&= \inner{f,c_{\lambda}} - u(\lambda)\inner{f, uc_{\lambda}} \\
		&= \big<f, (1 - \overline{u(\lambda)}u)c_{\lambda} \big>.
	\end{align*}
	Moreover, the function $(1 - \overline{u(\lambda)}u)c_{\lambda}$ belongs to $\K_u$ since
	\begin{align*}
		\big<uh,(1 - \overline{u(\lambda)}u)c_{\lambda} \big>
		&= u(\lambda)h(\lambda) - u(\lambda) \inner{uh,uc_{\lambda}} \\
		&= u(\lambda)h(\lambda) - u(\lambda) \inner{h,c_{\lambda}} \\
		&= u(\lambda)h(\lambda) - u(\lambda)h(\lambda) \\
		&= 0
	\end{align*}
	for all $h$ in $H^2$.  Putting this all together we see that
	\begin{equation*}
		f(\lambda) = \inner{f,k_{\lambda}}, \tag{$f \in \K_u$}
	\end{equation*}
	where
	\begin{equation}\label{eq:Kernel}
		k_{\lambda}(z) = \frac{1 - \overline{u(\lambda)}u(z)}{1 - \overline{\lambda}z}.
	\end{equation}
	The function $k_{\lambda}$ is called the \emph{reproducing kernel} for $\K_u$.	 
	When dealing with more than one model space at a time, the notation $k_{\lambda}^u$
	is often used to denote the reproducing kernel for $\K_u$.
	
	From the general theory of reproducing kernels \cite{MR0051437, Paulsen, AMC}, it follows that if $e_1,e_2,\ldots$ is
	any orthonormal basis for $\K_u$, then
	\begin{equation}\label{eq:KONB}
		k_{\lambda}(z) = \sum_{n\geqslant 1} \overline{e_n(\lambda)} e_n(z),
	\end{equation}
	the sum converging in the norm of $\K_u$ \cite{AMC}.
	The following example
	illustrates both expressions \eqref{eq:Kernel} and \eqref{eq:KONB}.

	\begin{Example}
		If $u = z^n$, then $\K_u = \operatorname{span}\{1,z,z^2,\ldots,z^{n-1}\}$ and
		\begin{equation*}
			k_{\lambda}(z) = \frac{1 - \overline{\lambda}^{n} z^n}{1 - \overline{\lambda} z} 
			= 1 + \overline{\lambda} z + \overline{\lambda}^2 z^2 + \cdots + \overline{\lambda}^{n - 1} z^{n -1}.
		\end{equation*}
	\end{Example}
	
	We are now in a position to provide an elegant derivation of the decomposition \eqref{eq:Kuv}. Indeed, for inner functions $u$ and $v$, divide the trivial equality 
	\begin{equation*}
		1 - \overline{u(\lambda)v(\lambda)}u(z)v(z)
		= \big(1 - \overline{u(\lambda)}u(z) \big) + \overline{u(\lambda)}u(z)\big( 1 - \overline{v(\lambda)}v(z) \big)
	\end{equation*}
	by $1 - \overline{\lambda}z$ to get the identity 
	\begin{equation*}
		k_{\lambda}^{uv} = k_{\lambda}^u + \overline{u(\lambda)}u(z) k_{\lambda}^v.
	\end{equation*}
	However, the preceding is simply a restatement, in terms of reproducing kernels, of the fact that
	$\K_{uv} = \K_u \oplus u\K_v$.

	Since 
	\begin{equation}\label{eq:KIHI}
		|k_{\lambda}(z)| \leqslant \frac{2}{1-|\lambda|}
	\end{equation}
	for each $z$ in $\D$,
	it follows that each $k_{\lambda}$ belongs to $\K_u \cap H^{\infty}$.
	Moreover, the kernel functions $k_{\lambda}$ are among the few readily identifiable functions
	that belong to $\K_u$.  As such, they provide invaluable insight into the structure and properties of model spaces.
	In fact, kernel functions often wind up being ``test functions'' for various statements about model spaces. 
	As an example of what we mean by this, suppose that the quantity
	\begin{equation*}
		\|k_{\lambda}\|^2 = \frac{1 - |u(\lambda)|^2}{1 - |\lambda|^2}
	\end{equation*}
	remains bounded as $\lambda \to \zeta \in \T$ along some path $\Gamma \subset \D$, then since 
	\begin{equation*}
		|f(\lambda)| = |\langle f, k_{\lambda} \rangle| \leqslant \|f\| \|k_{\lambda}\|,
	\end{equation*}
	we see that every $f$ in $\K_u$ is bounded along $\Gamma$. We will see more of these types of results in Section \ref{Section:BB}.

	The orthogonal projection 
	\begin{equation} \label{Pu-def}
		P_u:L^2\to\K_u
	\end{equation}
	 arises frequently in the study of Hankel operators, model operators and
	 truncated Toeplitz operators.  This important operator can be expressed in a simple
	 manner using reproducing kernels.
	
	\begin{Proposition}\label{Prop:Pu}
		For each $f$ in $L^2$ and $\lambda$ in $\D$,
		\begin{equation}\label{eq:Puf}
			(P_u f)(\lambda) = \inner{ f , k_{\lambda}}.
		\end{equation}
	\end{Proposition}
	
	\begin{proof}
		Using the selfadjointness of $P_u$ we get 
		\begin{equation*}
		\inner{f,k_{\lambda}} = \inner{f , P_uk_{\lambda}} = \inner{ P_u f, k_{\lambda}} = (P_u f)(\lambda). \qedhere
		\end{equation*}
	\end{proof}

	Rewriting \eqref{eq:Puf} as an integral we obtain
	\begin{equation*}
		(P_u f)(\lambda) = \int_{\T} f(\zeta) \frac{1 - u(\lambda) \overline{u(\zeta)}}{1 - \lambda \overline{\zeta}}\,dm(\zeta),
	\end{equation*}
	which can be used to estimate or control the behavior of $P_u f$.  In particular, the preceding formula
	highlights the dependence of $P_uf$ on the behavior of $u$ itself.

\subsection{Density results}
	As mentioned earlier, in general, the model spaces $\K_u$ contain few readily identifiable functions.  One therefore relies
	heavily upon the kernel functions $k_{\lambda}$ in the study of model spaces since they are
	among the few explicitly describable functions contained in $\K_u$.
	As the following proposition shows, one can always find a collection of kernels whose span is dense in the
	whole space.

	\begin{Proposition}\label{P-density-easy}
		If $\Lambda$ is a subset of $\D$ such that either
		(i) $\Lambda$ has an accumulation point in $\D$, or 
		(ii) $\sum_{\lambda \in \Lambda} \big(1-|\lambda|\big)$ diverges, 
		then for any inner function $u$
		\begin{equation*}
			\bigvee\{ k_{\lambda} : \lambda \in \Lambda \}  = \K_{u}.
		\end{equation*}
	\end{Proposition}

	\begin{proof}
		The containment $\subseteq$ is obvious.  
		If $f \perp k_{\lambda}$ for all $\lambda \in \Lambda$, then $f$ vanishes on $\Lambda$. 
		(i) If $\Lambda$ has an accumulation point in $\D$, then the Identity Theorem implies that $f \equiv 0$. 
		(ii) If $\sum_{\lambda \in \Lambda} (1-|\lambda|)$ diverges, then $f \equiv 0$ since the zero set of a nonzero 
		$H^2$ function must satisfy the Blaschke condition \eqref{eq:BlaschkeCondition}.
	\end{proof}	
	
	A somewhat more general result is provided by \cite[Thm.~1]{ICS}:
	
	\begin{Proposition}
		If $u$ is a nonconstant inner function, then there exists a subset $E \subset \D$ of area measure zero such that
		for each $w$ in $\D \backslash E$, the inverse image $u^{-1}(\{w\})$ is nonempty, $u'(\lambda)\neq 0$ for all
		$\lambda$ in $u^{-1}(\{w\})$, and
		\begin{equation*}
			\bigvee\big\{ k_{\lambda} : \lambda \in u^{-1}(\{w\}) \big\}  = \K_{u}.
		\end{equation*}
	\end{Proposition}
	
	\begin{Remark} We will take up the general question of whether a sequence of kernel functions has a dense linear span in $\K_u$ when we discuss completeness problems in Section \ref{CP}. 
	\end{Remark}
	
	Since each kernel function $k_{\lambda}$ belongs to $H^{\infty}$ by \eqref{eq:KIHI},
	an immediate consequence of either of the preceding propositions is the following useful result.

	\begin{Proposition}\label{P-H-dense}
		The linear manifold $\K_u \cap H^{\infty}$ is dense in $\K_u$.
	\end{Proposition}
	
	A second proof can be obtained by noting that
	\begin{equation*}
		\K_u = \bigvee \{S^{* n} u: n \geqslant 1\}
	\end{equation*}
	and that each of the backward shifts $S^{*n}u$ of $u$ belongs to $H^{\infty}$.
	In certain sophisticated applications one requires a dense subset of $\K_u$
	consisting of functions having a certain degree of smoothness on $\D^{-}$. 	
	This next result is a restatement of Proposition \ref{Prop:FDMS} for infinite Blaschke products. 
	
	\begin{Proposition}\label{Prop:Span}
		If $u$ is a Blaschke product with simple zeros $\lambda_1,\lambda_2,\ldots$, then 
		\begin{equation*}
			\bigvee\{ k_{\lambda_1}, k_{\lambda_2},\ldots\} = \K_u.
		\end{equation*}
		Thus $\K_u$ contains a dense subset whose elements are continuous on $\D^-$ and whose
		boundary functions are infinitely differentiable on $\T$.
	\end{Proposition}
		
	\begin{proof}
		Indeed, suppose that $f$ is a function in $\K_u$ that satisfies $\inner{f,k_{\lambda_n}} = 0$ for all $n$.
		This implies that $f(\lambda_n) = 0$ for all $n$,
		whence $u|f$ so that $f$ belongs to $uH^2$.  In other words, $f$ must be identically zero.
		Since
		\begin{equation}\label{k-c}
			k_{\lambda_n}(z)  = \frac{1 - \overline{u(\lambda_n)}u(z)}{1 - \overline{\lambda_n}z} 
			= \frac{1}{1 - \overline{\lambda_n}z}  = c_{\lambda_n}(z),
		\end{equation}
		the second statement follows immediately.
	\end{proof}
		
	If the zeros of $u$ are not simple, then the preceding statement is true if one also includes the appropriate 
	derivatives of the kernel functions in the spanning set.  When $u$ is not a Blaschke product, finding 
	a dense set of functions in $\K_u$, each of which is continuous on $\D^-$, is much more difficult.  A deep result in this
	direction is due to A.B.~Aleksandrov \cite{MR1359992}, who proved the following astonishing result (see \cite{CMR} for a discussion of this).

	\begin{Theorem}[Aleksandrov]
		For an inner function $u$, the set $\K_u \cap \mathcal{A}$ is dense in $\K_u$.
		Here $\mathcal{A}$ denotes the disk algebra, the Banach algebra of all $H^{\infty}$ functions that are
		continuous on $\D^-$.
	\end{Theorem}

	This theorem is remarkable due to the fact that $\K_u$ often does not contain a single readily identifiable
	function that is continuous on $\D^-$.  For example, if $u$ is the singular inner function
	\begin{equation*}
		u(z) = \exp\left(\frac{z + 1}{z - 1}\right),
	\end{equation*}
	then it is not at all obvious that $\K_u$ contains any functions that are continuous on $\D^{-}$, 
	let alone a dense set of them.  See \cite{MR2198372} for some related results concerning when 
	$\K_u$ contains smoother functions than those in $\mathcal{A}$. 

\subsection{Cauchy-Szeg\H{o} bases}\label{Subsection:Riesz}

	If $u$ is an inner function and $u(\lambda)=0$, then \eqref{k-c} tells us that
	$k_{\lambda}(z) = c_{\lambda}(z) = (1 - \overline{\lambda} z)^{-1}$.  Thus the kernel functions corresponding
	to zeros of $u$ are extremely simple functions to work with since they do not
	depend explicitly upon $u$.  In certain situations, these functions can be used
	to construct bases for models spaces (in a sense to be made precise shortly).
	However, no two Cauchy kernels are orthogonal, so one 
	cannot hope to obtain an orthonormal basis of kernel functions.\footnote{One can sometimes obtain 
		orthonormal bases consisting of \emph{boundary kernels} (see Section \ref{Section:Clark}).}
	We therefore need a somewhat more flexible definition \cite{MR1946982}.

		\begin{Definition}
			A linearly independent sequence $\vec{x}_n$ in a Hilbert space $\h$ is called a \emph{Riesz basis} for $\h$ if 
			$\bigvee\{ \vec{x}_1, \vec{x}_2,\ldots\} = \h$ and there
			exist constants $M_1,M_2>0$ such that
			\begin{equation*}
				M_1 \sum_{i=1}^n |a_i|^2 \leqslant \norm{ \sum_{i=1}^n a_i \vec{x}_i}^2 \leqslant M_2 \sum_{i=1}^n |a_i|^2
			\end{equation*}
			for all finite numerical sequences $a_1,a_2,\ldots,a_n$.
		\end{Definition}
		
	The following seminal result of L.~Carleson tells us when the normalized
	reproducing kernels corresponding to the zero set of a Blaschke product forms a Riesz basis for $\K_u$ \cite[p.~133]{N1}
	
	\begin{Theorem}\label{Theorem:Interpolation}
		If $u$ is a Blaschke product with simple zeros $\lambda_1,\lambda_2,\ldots$, then the normalized kernels
		$$\frac{k_{\lambda_n}}{\norm{k_{\lambda_n}}} = \frac{ \sqrt{1 - |\lambda_n|^2}}{1 - \overline{\lambda_n}z}$$ 
		form a Riesz basis for $\K_u$ if and only if there exists a $\delta > 0$ such that
		\begin{equation}
			\delta <  \prod_{\substack{j=1\\ j \neq i} }^{\infty} \left|\frac{\lambda_i - \lambda_j}{1 - \overline{\lambda_j} \lambda_i} \right|.
			\tag{$i=1,2,\ldots$}
		\end{equation}
	\end{Theorem}

There is a generalization of this result where one can somewhat relax the condition that the $\lambda_n$ are the zeros of $u$ \cite{HNP}.

\subsection{Takenaka-Malmquist-Walsh bases}\label{Subsection:TMW}

	Unlike the Hardy space $H^2$ itself, the model spaces $\K_u$ do not come pre-equipped with a
	canonical orthonormal basis.  It turns out that an orthonormal basis for $\K_u$, where $u$ is a Blaschke product,
	can be obtained by orthogonalizing the kernel functions corresponding to the zeros of $u$.

	For $w$ in $\D$, we let
	\begin{equation}\label{eq:b}
		b_w(z) = \frac{z-w}{1-\overline{w}z}.
	\end{equation}	
	If $u$ is a Blaschke product with simple zeros $\lambda_1,\lambda_2,\ldots$, then observe that
	\begin{equation*}
		\inner{b_{\lambda_1}k_{\lambda_2},k_{\lambda_1}}  = b_{\lambda_1}(\lambda_1)k_{\lambda_2}(\lambda_1)
		 = 0.
	\end{equation*}
	Similarly, we have
	\begin{equation}
		\inner{b_{\lambda_1}b_{\lambda_2}k_{\lambda_3},k_{\lambda_1}}  = 
		\inner{b_{\lambda_1}b_{\lambda_2}k_{\lambda_3},k_{\lambda_2}}  = 0.
	\end{equation}
	This process suggests the following iterative definition:
	If $u$ is a Blaschke product with zeros $\lambda_1,\lambda_2,\ldots$, let
	$$v_{1}(z) := \frac{\sqrt{1 - |\lambda_1|^2}}{1 - \overline{\lambda_1} z}, 
	$$
	\begin{equation*}v_k(z) =
		\Big( \prod_{i=1}^{k-1} b_{\lambda_i} \Big)
		\dfrac{ \sqrt{1 - |\lambda_k|^2} }{ 1 - \overline{\lambda_k}z} . \tag{$k\geqslant 2$}
	\end{equation*}
	One can show that $\{v_n: n \geqslant 1\}$ is an orthonormal basis for $\K_u$ that is 
	the Gram-Schmidt orthonormalization of the kernels $k_{\lambda_1},k_{\lambda_2}, \ldots$.  

	The terminology here is not completely standard.  It seems that these bases first appeared
	in Takenaka's 1925 paper involving finite Blaschke products \cite{Tak}.  The classic text \cite{N1}
	considers the general case where $u$ is a potentially infinite Blaschke product,
	referring to this result as the Malmquist-Walsh Lemma.  In light of Takenaka's
	early contribution to the subject, the authors chose in \cite{NLEPHS} to refer to such a basis for
	$\K_u$ as a \emph{Takenaka-Malmquist-Walsh (TMW) basis}. 

	Another important family of orthonormal bases of $\K_u$ is (sometimes) provided by the Aleksandrov-Clark spectral theory of
	rank-one unitary perturbations of model operators.
	The so-called \emph{modified Aleksandrov-Clark bases} are particularly important in the study of
	finite-dimensional model spaces.
	To discuss these bases, we first require a few words about boundary behavior and angular derivatives.


\section{Boundary behavior}\label{Section:BB}

\subsection{Pseudocontinuations} 
		Functions in model spaces enjoy certain ``continuation'' properties across $\T$. 
		These types of continuation properties appear in many places \cite{RS} and in many different settings 
		but they all go under the broad heading of ``generalized analytic continuation.'' 
		The type of generalized analytic continuation that is relevant to model spaces is called \emph{pseudocontinuation} 
		and it was first explored by H. S.~Shapiro \cite{MR0241614, MR0267080}.				
		In what follows, we let 
		$$\mathbb{D}_e := \big\{ |z|>1\big\} \cup \{\infty\}$$ denote the {\em extended exterior disk}, the complement of the closed unit
		disk in the extended complex plane.
		
		\begin{Definition}
		Let $f$ and $\widetilde{f}$ be meromorphic functions on $\mathbb{D}$ and $\mathbb{D}_e$, respectively. 
		If the nontangential limiting values of $f$ (from $\D$) agree with the nontangential limiting values of $\widetilde{f}$ (from $\D_e$) almost everywhere on $\T$, then we say that $f$ and $\widetilde{f}$ are  \emph{pseudocontinuations} of one another.  
		\end{Definition}

\begin{Remark}
Pseudocontinuations are unique in the sense that if $F$ and $G$ are meromorphic functions on $\D_e$ that are both pesudocontinuations of $f$, then $F = G$. This follows from the Privalov uniqueness theorem \cite[p.~62]{Koosis}, which states that if $f$ and $g$ are meromorphic on $\D$ with equal non-tangential limits on any set of positive Lebesgue measure on $\T$, then $f = g$. This is why the definition of pseudocontinuation is stated in terms of nontangential limits as opposed to radial limits, where Privalov's theorem is no longer true. However, in the context we will apply this definition, all the functions involved will be of {\em bounded type}, that is to say the quotient of two bounded analytic functions (recall Definition \ref{bounded-type}), where the nontangential limits exist almost everywhere. 

By Privalov's uniqueness theorem again, we can show that pseudocontinuation is compatible with analytic continuation in that if $f$ (on $\D$) has a pseudocontinuation $F$ (on $\D_e$) and $f$ also has an analytic continuation $F_1$ across some neighborhood of a point on $\T$, then $F_1 = F$. 
\end{Remark}

\begin{Example}\label{PCBT-ex}
Let us provide a few instructive examples (see \cite{RS} for further details).
\begin{enumerate}\addtolength{\itemsep}{0.25\baselineskip}
\item Any inner function $u$ has a pseudocontinuation to $\D_e$ defined by 
\begin{equation} \label{utilde}
\widetilde{u}(z) := \frac{1}{\,\overline{u(1/\overline{z})}\,}.
\end{equation}
 In fact, this pseucontinuation is one of bounded type (written PCBT), being a quotient of two bounded analytic functions on $\D_e$. Pseudocontinuations of bounded type will play an important role momentarily. 
\item If $f$ is a rational function whose poles lie in $\D_e$, then $f$ is PCBT.
\item Since pseudocontinuations must be compatible with analytic continuations (see the previous remark), functions with isolated branch points on $\T$, such as $\sqrt{1-z}$, do not possess pseudocontinuations. 
\item The function $f(z) = \exp z$ is not pseudocontinuable.   
			Although it is analytically continuable to $\C$, it is not meromorphic on $\mathbb{D}_e$ due to its essential singularity at $\infty$.
\item The classical gap theorems of Hadamard \cite{Hadamard2} and Fabry \cite{Fabry} show that analytic functions on $\D$ with lacunary power series do not have analytic continuations across any point of the unit circle. Such functions do not have pseudocontinuations either \cite{Abakumov-96, MR1458416, MR0267080}.
\end{enumerate}
\end{Example}

The following seminal result demonstrates the direct connection between membership in a model space
		and pseudocontinuability. In what follows, 
		\begin{equation} \label{H2De}
		H^2(\D_e) := \{f(1/z): f \in H^2\}
		\end{equation} denotes the Hardy space of the extended exterior disk.

		\begin{Theorem}[Douglas, Shapiro, Shields \cite{DSS}]\label{DSS-PC}
		A function $f \in H^2$ belongs to $\K_u$ if and only if $f/u$ has a pseudocontinuation $F_u \in H^2(\D_e)$ satisfying $F_u(\infty) = 0$. 
		\end{Theorem}
		
		\begin{proof}
		Suppose that $f$ belongs to $\K_u$. By Proposition \ref{Prop:fgzu}, $f =  \overline{g\zeta}u$ almost everywhere on $\T$. Define the function $F_u$ on $\D_e$ by 
		\begin{equation}\label{PC-H-D}
		F_{u}(z) := \frac{1}{z} \overline{g\Big(\frac{1}{\,\overline{z}\,}\Big)}
		\end{equation} and note by \eqref{H2De} that $F_u$ is in $H^2(\D_e)$ and $F_{u}(\infty) = 0$. Moreover, by the identity $f =  \overline{g\zeta}u$ almost everywhere on $\T$ we see that $F_u(\zeta) = f(\zeta)/u(\zeta)$ for almost every $\zeta \in \T$, i.e., $F_u$ is a pseudocontinuation of $f/u$. This proves one direction. 
		
		For the other direction, suppose that $f/u$ has a pseudocontinuation $F_u \in H^2(\D_e)$ with $F_{u}(\infty) = 0$. Then, again, by \eqref{H2De} we have 
		$$F_{u}(z) = \frac{1}{z} h\Big(\frac{1}{z}\Big)$$
		for some $h \in H^2$ and 
		$f(\zeta)/u(\zeta) = F_{u}(\zeta)$ for almost every $\zeta$ in $\T$. Define a function $g$ on $\D$ by
		$g(z) := \overline{h(\overline{z})}$ and observe that $g \in H^2$. From here we see that $f(\zeta) = u(\zeta) \overline{\zeta g(\zeta)}$ for almost every $\zeta$ in $\T$ from which it follows, via Proposition \ref{Prop:fgzu}, that $f$ belongs to $\K_u$. 
		\end{proof}
		
With a little more work, one can also determine the {\em cyclic vectors} for $S^{*}$, those $f$ in $H^2$ for which 
$$\bigvee\{S^{* n} f: n = 0, 1, 2, \cdots\} = H^2.$$

\begin{Corollary}
A function $f$ in $H^2$ is not cyclic for $S^{*}$ if and only if $f$ is $PCBT$. 
\end{Corollary}

\begin{proof}
We will prove only one direction. Indeed, suppose $f$ is non-cyclic for $S^{*}$. Then $f$ must belong to some model space $\K_u$. By the previous theorem, $f/u$ has a pseudocontinuation $F_u$ that belongs to $H^2(\D_e)$. But since $u$ already has a natural pseudocontinuation $\widetilde{u}$ given by \eqref{utilde}, we see that $f$ has a pseudocontinuation $\widetilde{u} F_u$ that is of bounded type\footnote{We are using the well-known fact here that $F_u$ (or any function in $H^2(\D_e)$) is of bounded type (see Remark \ref{N-facts}).}. 
\end{proof}

Along with Example \ref{PCBT-ex} this corollary shows that inner functions and rational function are noncyclic for $S^{*}$ whereas functions like $\sqrt{1 - z}$ and $e^z$ are cyclic. 
		
		In principle the cyclic vector problem for $S^{*}$ is solved (i.e., $f$ is noncyclic for $S^{*}$ if and only if $f$ is PCBT). 
		However, this solution is not as explicit as the solution to the cyclic vector problem for $S$ (i.e., $f$ is cyclic for $S$ if and only $f$ is outer). 
		Outer functions are, in a sense, readily identifiable.  On the other hand, PCBT functions are not so easily recognized. 
		We refer the reader to some papers which partially characterize the noncyclic vectors 
		for $S^{*}$ by means of growth of Taylor coefficients \cite{Kriete-1970, Shapiro-64}, 
		the modulus of the function \cite{Kriete-1970, DSS}, and gaps in the Taylor coefficients \cite{Abakumov-96, MR1458416, DSS}.

		\subsection{Analytic continuation}
		Recall from \eqref{eq:SpecDef} and Theorem \ref{TheoremSpectrum} that
		if $u = B_{\Lambda} S_{\mu}$, where $B_{\Lambda}$ is a Blaschke product with zero set
		$\Lambda$ and $S_{\mu}$ is a singular inner function with corresponding
		singular measure $\mu$, then the \emph{spectrum} of $u$ is the set
		\begin{equation*}
			\sigma(u) =\Lambda^{-} \cup \operatorname{supp} \mu .
		\end{equation*}
		The relevance of $\sigma(u)$ to the function theoretic properties of $\K_u$
		lies in the following observation \cite[p.~65]{N1}, \cite[p.~84]{CR}.
				
		\begin{Proposition} \label{P-analytic-continuation}
			Every $f$ in $\K_u$ can be analytically continued to 
			\begin{equation*}
				\widehat{\C} \backslash \{ 1/\overline{z} : z\in \sigma(u)\},
			\end{equation*}
			where $\widehat{\C} = \C \cup \{\infty\}$ denotes the extended complex plane.		
		\end{Proposition}
		
		\begin{proof}
		A sketch of the proof goes as follows.   By \cite[Sect.~2.6]{Garnett} $u$ has an analytic continuation across $\T \setminus \sigma(u)$. By Theorem \ref{DSS-PC}, $f/u$ has a pseudocontinuation $F_u$ in $H^2(\D_e)$ for any $f \in \K_u$. 
		Let $J$ be an open arc contained in $\T \setminus \sigma(u)$. We know that $u$ has an analytic continuation across $J$. We also know that the integral means 
		$$\int_{J} \Big|\frac{f}{u}(r \zeta)\Big| dm(\zeta)$$
and 
$$\int_{J} |F_u(s \zeta)| dm(\zeta)$$ are uniformly bounded in $0 < r < 1$ and $s > 1$. One can now take a triangle $\Delta$ cut by $J$ and show (by cutting $\Delta$ with $(1 - \epsilon) J$ and $(1 + \epsilon)J$ and using the boundedness of the above integral means) that the contour integral of the function defined by $f/u$ on $\D$ and $F_u$ on $\D_e$ integrates to zero on $\Delta$. Now apply Morera's theorem\footnote{There is a precise version of this type of Morera's theorem found in \cite{Garnett}.}. 
		\end{proof}
		
		\subsection{Nontangential limits}\label{Subsection:Limits}
		Functions in $H^2$ possess nontangential limits almost everywhere on $\T$. However, one can easily see that for each fixed $\zeta$ in $\T$ there exists an $H^2$ function 
		that does not have a finite nontangential limit at $\zeta$ (e.g., $f(z) = \log(\zeta - z)$). 
		In sharp contrast to this, for a given model space $\K_u$, there might exist
		 points $\zeta$ in $\T$ such that each function in $\K_u$ possesses a nontangential limit at $\zeta$. We begin with a definition. 
		
		\begin{Definition}
			If $u$ is inner and $\zeta \in \T$, then $u$ has an
			\emph{angular derivative in the sense of Carath\'eodory} (ADC)
			\emph{at $\zeta$} if the nontangential limits of $u$ and $u'$ exist at $\zeta$ and  $|u(\zeta)| = 1$.
		\end{Definition}
			
		A number of equivalent conditions for the existence of an angular derivative (called the Julia-Carath\'eodory theorem) can be found in \cite[p.~57]{JShap}.
		The precise relationship between angular derivatives and model spaces is contained in the following
		important result.
		
		\begin{Theorem}[Ahern-Clark \cite{AC70, AC70a}]\label{TheoremAC}
			For an inner function $u = B_{\Lambda} S_{\mu}$, where $B_{\Lambda}$
			is a Blaschke product with zeros $\Lambda = \{\lambda_n\}_{n=1}^{\infty}$, repeated
			according to multiplicity,  $S_{\mu}$ is a singular inner function
			with corresponding singular measure $\mu$,  and $\zeta \in \T$, the following are equivalent:
			\begin{enumerate}\addtolength{\itemsep}{0.75\baselineskip}
				\item Every $f \in \mathcal{K}_{u}$ has a nontangential limit at $\zeta$.
				\item For every $f \in \mathcal{K}_u$, $f(\lambda)$ is bounded as $\lambda \to \zeta$ nontangentially.
				\item $u$ has an ADC at $\zeta$.
				\item $\displaystyle k_{\zeta}(z) = \frac{1 - \overline{u(\zeta)} u(z)}{1 - \overline{\zeta} z} \in H^2,$
				\item $\displaystyle
						\sum_{n \geqslant 1} \frac{1 - |\lambda_{n}|^2}{|\zeta - \lambda_{n}|^2} +
						\int_{\T} \frac{d \mu(\xi)}{|\xi - \zeta|^2} < \infty.$
			\end{enumerate}
		\end{Theorem}
			
		It is worth comparing condition (v) of the preceding theorem to an old result of Frostman \cite{Frostman}
		that says that an inner function $u$ (and all its divisors) has a non-tangential limits of modulus one at $\zeta$ in $\T$
		whenever
		\begin{equation*}
			\sum_{n \geqslant 1} \frac{1 - |\lambda_{n}|}{|\zeta - \lambda_{n}|} +
			\int_{\T} \frac{d \mu(\xi)}{|\xi - \zeta|} < \infty.
		\end{equation*}
		We should also mention that there are inner functions (in fact Blaschke products) 
		$u$ for which $|u'(\zeta)| = \infty$ at every point of $\T$ \cite{Frostman}. 
		
		There are also results about \emph{tangential} limits of functions from model spaces \cite{Berman, Cargo}. 
		When condition (v) fails then there are functions in $\K_u$ that do not have nontangential limits at $\zeta$. 
		In particular, there are functions $f$ in $\K_u$ for which $|f(r \zeta)|$ is unbounded as $r \to 1^{-}$. However, in \cite{Hartmann-Ross} there are results that give bounds (both upper and lower) on the growth of $|f(r \zeta)|$ as a function of $r$. 		
				
\section{Conjugation}\label{SectionConjugation}		
	Each model space $\K_u$ comes equipped with a conjugation, a certain type of conjugate-linear operator 
	that generalizes complex conjugation on $\C$.  Not only does this conjugation cast a new light
	on pseudocontinuations, it also interacts with a number of important linear operators that act upon model spaces.  Most of the material in
	this section can be found in \cite{CCO,CBSTK}.

\subsection{Basic properties}
	We say that a conjugate-linear function $C:\h\to\h$ on a complex Hilbert space $\h$ is a \emph{conjugation}
	if $C$ is isometric (i.e., $\norm{Cx} = \norm{x}$ for all $x$ in $\h$) and involutive (i.e., $C^2 = I$).
	In light of the polarization identity, it turns out that the isometric condition is equivalent to the assertion that
	$\inner{ Cx,Cy}=\inner{y,x}$ for all $x, y$ in $\h$.

	The structure of conjugations is remarkably simple.  If $C$ is a conjugation on $\h$, then there exists
	an orthonormal basis $\{ e_n\}$ of $\h$ such that $Ce_n = e_n$ for all $n$  \cite[Lem.~1]{G-P}.  We refer to such a basis as
	a \emph{$C$-real} orthonormal basis of $\h$.  As a consequence, any conjugation is unitarily equivalent to  
	the \emph{canonical conjugation}
	\begin{equation*}
		(z_1,z_2,\ldots) \mapsto  (\overline{z_1},\overline{z_2},\ldots)
	\end{equation*}
	on an $\ell^2$-space of the appropriate dimension.
	
	In order to understand the natural conjugation on $\K_u$, we must first remind the reader that $\K_u = H^2 \cap u \overline{z H^2}$,
	as a space of functions on $\T$ (see Proposition \ref{Prop:fgzu}).
	
	\begin{Proposition}\label{Prop:C}
		The conjugate-linear operator $C:\K_u\to\K_u$, defined in terms of boundary functions on $\T$
		by $Cf = \overline{fz}u$, is a conjugation.  
		In particular, $|f| = |Cf|$ almost everywhere on $\T$ so that $f$ and $Cf$ share the same outer factor.
	\end{Proposition}

	\begin{proof}
		Since $|u|=1$ almost everywhere on $\T$, it follows that $C$ is conjugate-linear, isometric, and involutive.
		We need only prove that $C$ maps $\K_{u}$ into $\K_{u}$. Since $f$ is orthogonal to $uH^2$,
		it follows that
		\begin{align*}
			 \inner{Cf,\overline{zh}}  
			 &= \inner{ \overline{fz}u, \overline{zh} } \\
			 &= \int_{\T} \overline{f(\zeta)\zeta}u(\zeta) \zeta h(\zeta)\,dm(\zeta) \\
			 &= \int_{\T} u(\zeta) h(\zeta)\overline{f(\zeta)}\,dm(\zeta) \\
			&= \inner{u h ,f}  \\
			&= 0
		\end{align*}
		for all $h$ in $H^2$.  In other words, $Cf$ belongs to $H^2$.   Similarly,
		\begin{equation*}
			\inner{Cf, u h}  = \inner{ \overline{fz}u, u h} = \inner{ \overline{fz}, h} 
			= 0, 
		\end{equation*}
		from which it follows that $Cf$ belongs to $\K_{u}$.  
	\end{proof}

	\begin{Example}
		If $u(z) = z^n$, then $\K_u = \bigvee\{1, z, \ldots, z^{n - 1}\}$ and the conjugation $C$ assumes the form 
		\begin{equation*}
			C(a_0 + a_1 z + \cdots a_{n - 1} z^{n - 1})  =  \overline{a_{n - 1}}+ \overline{a_{n-2}} z +  \cdots + \overline{a_0} z^{n-1}.
		\end{equation*}
		Now suppose that $u$ is a finite Blaschke product having zeros $\lambda_1,\lambda_2,\ldots,\lambda_n$,
		repeated according to multiplicity.  Referring back to Proposition \ref{Prop:FDMS} for an explicit description of $\K_u$, 
		one can show that
		\begin{equation*}
			C \left( \frac{a_0 + a_1 z + \cdots +a_{n-1} z^{n-1}}{ \prod_{i=1}^n(1-\overline{\lambda_i}z)} \right)
			=\frac{ \overline{a_{n-1}} + \overline{a_{n-2}} z + \cdots +\overline{a_0} z^{n-1}}{ \prod_{i=1}^n(1-\overline{\lambda_i}z)} .
		\end{equation*}
		In other words, each element of $\K_u$ can be represented as a rational function whose denominator is
		$\prod_{i=1}^n(1-\overline{\lambda_i}z)$ and whose numerator is a polynomial of degree $\leqslant n-1$.
		The conjugation $C$ acts on $\K_u$ by reversing the order of the polynomial in the numerator and conjugating its coefficients.
	\end{Example}

	\begin{Example}\label{Example:DQ}
		A straightforward computation reveals that $C$ sends reproducing kernels to difference quotients:
		\begin{align*}
			[Ck_{\lambda}](z) 
			&=\overline{\left( \frac{1 - \overline{u(\lambda)}u(z)}{1 - \overline{\lambda}z} \right) } \overline{z}u(z)   \\
			&=\frac{1 - u(\lambda) \overline{u(z)}}{1 - \lambda\overline{z}}\cdot \frac{u(z)}{z}  \\
			&= \frac{ u(z) - u(\lambda) }{z - \lambda}, 
		\end{align*}
		where, as usual, we consider all of the functions involved as functions on $\T$ (i.e., so that $z\overline{z}=1$).
		For each $\lambda$, the kernel function $k_{\lambda}$ is an outer function since it is the quotient of the two
		outer functions $1 - \overline{u(\lambda)}u(z)$ and $1 - \overline{\lambda}z$ \cite[Ex.~1, Ch.~3]{Duren}.
		In light of Proposition \ref{Prop:C}, we expect that the difference quotient $Ck_{\lambda}$ is simply an inner multiple
		of $k_{\lambda}$.  This is indeed the case as the following computation shows.
		\begin{align*}
			[Ck_{\lambda}](z)	
			&=	\frac{ u(z) - u(\lambda)}{z - \lambda}\\
			&=	\frac{ u(z) - u(\lambda)}{1 - \overline{u(\lambda)}u(z)}\cdot \frac{1 - \overline{\lambda}z}{z - \lambda}
				\cdot \frac{1 - \overline{u(\lambda)}u(z)}{1 - \overline{\lambda}z}\\
			&=	\frac{b_{u(\lambda)}(u(z))}{b_{\lambda}(z)} k_{\lambda}(z),
		\end{align*}
		where $b_w(z) = (z-w)/(1-\overline{w}z)$ is the disk automorphism \eqref{eq:b}.
	\end{Example}

	\begin{Example}\label{Example:BK}
		If $\zeta$ is a point on $\T$ for which $u$ has an ADC (see Subsection \ref{Subsection:Limits}), then
		the boundary kernel $k_{\zeta}$, as defined in Theorem \ref{TheoremAC}, belongs to $\K_u$.  In this case,
		something remarkable occurs.  Since $\zeta$ and $u(\zeta)$ are of unit modulus, it follows that
		\begin{align*}
			[Ck_{\zeta}](z)	
			&=	\frac{ u(z) - u(\zeta) }{z - \zeta}\\
			&=	\frac{u(\zeta)}{\zeta} \cdot \frac{1 - \overline{u(\zeta)}u(z)}{1 - \overline{\zeta}z}\\
			&=	\overline{\zeta} u(\zeta) u_{\zeta}(z).
		\end{align*}
		For either branch of the square root, it follows that the outer function
		\begin{equation*}
			(\overline{\zeta} u(\zeta))^{\frac{1}{2}} k_{\zeta}
		\end{equation*}
		belongs to $\K_u$ and is fixed by $C$.  
	\end{Example}

	Under certain circumstances, it is possible to construct $C$-real orthonormal bases for $\K_u$ 
	using boundary kernels.  This is relatively straightforward for finite dimensional model spaces, as the following
	example demonstrates (see Section \ref{Section:Clark} for a discussion of the general setting).

\begin{Example}
	Let $u$ be a finite Blaschke product with $n$ zeros repeated according to multiplicity.
	For the sake of simplicity, suppose that $u(0) = 0$.
	For each fixed $\alpha$ in $\T$, the equation $u(\zeta) = \alpha$ has precisely $n$ distinct solutions
	$\zeta_1, \zeta_2, \ldots, \zeta_n$ on $\T$.  The functions 
	\begin{equation*}
		e_j(z) =  \frac{e^{\frac{i}{2}(\arg \alpha - \arg \zeta_j ) }}{\sqrt{|u'(\zeta_j)|}}\frac{1 - \overline{\alpha} u(z)}{1 - \overline{\zeta_j} z}
	\end{equation*}
	for $j=1,2,\ldots, n$ form a $C$-real orthonormal basis of $\K_{u}$.
\end{Example}

\subsection{Associated inner functions}
	Fix an inner function $u$ and let $f = I_1 F$ denote the inner-outer factorization of a nonzero function $f$ in $\K_u$.
	Proposition \ref{Prop:C} ensures that $g = Cf$ has the same outer factor as $f$ whence we may write
	$g = I_2 F$ for some inner function $I_2$. 
	Since $g = \overline{fz}u$ almost everywhere on $\T$ it follows that $I_2 F = \overline{I_1 F z}u$ whence
	\begin{equation}\label{eq:IfIgF}
		I_1 I_2  = \frac{\overline{Fz}u}{F}.
	\end{equation}
	The inner function $\mathcal{I}_F = I_1 I_2$ is called the \emph{associated inner function} of $F$ with respect to $u$ \cite{CBSTK}.\footnote{We should note 
		that the influential article \cite[Rem.~3.1.6]{DSS} of Douglas, Shapiro, and Shields refers to $u$ itself as the associated
		inner function of $f$.}	
	In light of \eqref{eq:IfIgF}, we see that $\mathcal{I}_F$ is uniquely determined by $u$ and $F$.
	
	On the other hand, if $I_1$ and $I_2$ are inner functions such that $I_1 I_2 = \mathcal{I}_F$, then \eqref{eq:IfIgF} implies that
	\begin{equation*}
		I_1 F = \overline{I_2 F z}u
	\end{equation*}
	so that the functions $f = I_1 F$ and $g = I_2 F$ satisfy $f = \overline{gz}u$ almost everywhere on $\T$.
	By Proposition \ref{Prop:fgzu}, it follows that $f$ and $g$ belong to $\K_u$ and satisfy $Cf = g$.	
	Putting this all together, we obtain the following result.

	\begin{Proposition}
		The set of all functions in $\K_u$ having outer factor $F$ is precisely 
		\begin{equation}\label{eq:Poset}
			\{ \theta F : \text{$\theta$ inner and $\theta | \mathcal{I}_F$}\}. 
		\end{equation}
		We define a partial ordering on \eqref{eq:Poset} by declaring that
		$\theta_1 F \leq \theta_2 F$ if and only if $\theta_1 | \theta_2$.  With respect to this ordering, 
		$F$ and $\mathcal{I}_F F$ are the unique minimal and maximal elements, respectively.
		Moreover, $C$ restricts to an order-reversing bijection from \eqref{eq:Poset} to itself.
	\end{Proposition}
	
\subsection{Generators of $\K_u$}

	We say that a function $f$ in $\K_u$ \emph{generates $\K_u$} if 
	\begin{equation}\label{eq:fSSS}
		\bigvee\{f,S^*f,S^{*2}f,\ldots\} = \K_u.
	\end{equation}
	Recall that Proposition \ref{Prop:Generator} tells us that if $f = \overline{gz}u$ belongs to $\K_u$,
	then \eqref{eq:fSSS} holds if and only if
	$u$ and the inner factor of $g$ are relatively prime.  Recognizing that $g = Cf$ in this setting
	immediately yields \cite[Prop.~4.3]{CBSTK}.

	\begin{Proposition}
		If $f$ belongs to $\K_u$ and the inner factor of $Cf$ is relatively prime to $u$, then
		$f$ generates $\K_u$.
	\end{Proposition}
	
	\begin{Example}
		Each difference quotient $Ck_{\lambda} = (u(z)-u(\lambda))/(z-\lambda)$ generates $\K_u$
		since its conjugate $C(Ck_{\lambda}) = k_{\lambda}$ is outer.
		If $u$ is a singular inner function (i.e., $u$ is not divisible by any Blaschke product), then
		Frostman's Theorem \cite[Ex.~8, Ch.~2]{Duren}, \cite{Frostman} asserts that $k_{\lambda}$ generates $\K_u$ for almost every (with respect to area measure) $\lambda$ in $\D$.
		Indeed, Example \ref{Example:DQ} tells us that the inner factor of $Ck_{\lambda}$ is precisely 
		$b_{u(\lambda)}(u(z)))/b_{\lambda}(z)$, which is a Blaschke product whenever $u(\lambda)$ is nonzero and does not
		lie in the exceptional set for $u$.
	\end{Example}

	\begin{Corollary}
		If $f$ is an outer function in $\K_u$, then $Cf$ generates $\K_u$.  In particular, 
		any self-conjugate outer function in $\K_u$ generates $\K_u$.  
	\end{Corollary}

	\begin{Example}
		Any boundary kernel $k_{\zeta}$ in $\K_u$ generates $\K_u$.  Indeed, Example \ref{Example:BK}
		shows that any such function is a constant multiple of a self-conjugate outer function.
	\end{Example}
	
\subsection{Quaternionic structure of $2 \times 2$ inner functions}	
	It at least one instance, it turns out that conjugations on model spaces can yield structural insights into the nature
	of higher order inner functions.  The following result can be found in \cite{IMDS}, although it is stated in
	\cite{G-P} without proof, along with several other related results.

	\begin{Theorem}
		If $u$ is a nonconstant inner function, then the function $\Theta:\D\to M_2(\C)$ defined by
		\begin{equation}\label{eq:MIF}
			\Theta = \minimatrix{a}{-b}{Cb}{Ca} 
		\end{equation}
		is unitary almost everywhere on $\T$ and satisfies $\det \Theta = u$ if and only if 
		\begin{enumerate}\addtolength{\itemsep}{0.5\baselineskip}
			\item	$a,b,c,d$ belong to $\K_{zu}$. 
			\item	$|a|^2 + |b|^2 = 1$ almost everywhere on $\T$.
			\item	$Ca = d$ and $Cb = c$.
		\end{enumerate}
		Here $C:\K_{zu}\to\K_{zu}$ is the conjugation $Cf = \overline{f}u$ on $\K_{zu}$.
	\end{Theorem}
	
	We remark that \eqref{eq:MIF} is entirely analogous to the representation of quaternions of unit modulus
	using $2 \times 2$ complex matrices.  This representation was exploited in \cite{Chevrot} to study the
	characteristic function of a complex symmetric contraction.

\subsection{Cartesian decomposition}
	Each $f$ in $\K_u$ enjoys a \emph{Cartesian decomposition} $f=a+ib$ where the functions
	\begin{equation*}
		a = \tfrac{1}{2}(f + Cf),\qquad
		b=\tfrac{1}{2i}(f - Cf),
	\end{equation*}
	are both fixed by $C$.  With respect to this decomposition the conjugation $C$ on $\K_u$
	assumes the form $Cf = a-ib$.	
	To describe the functions that belong to $\K_u$, it suffices to describe those functions that
	are \emph{$C$-real} (i.e., such that $Cf = f$).  
	
	If $f = Cf$, then
	\begin{equation*}
		f = \overline{fz} u.
	\end{equation*}
	Since $u$ has unit modulus a.e.~on $\T$, by replacing $u$ with a suitable unimodular constant
	multiple of $u$ (which does not change the model space $\K_u$), we may assume that $u(\zeta) = \zeta$ for some $\zeta$ in $\T$. 

	If $u$ has an ADC at $\zeta$, then Theorem \ref{TheoremAC} tells us that the boundary kernel $k_{\zeta}$
	belongs to $\K_u$.  In the event that $k_{\zeta}$ fails to belong to $\K_u$ (i.e., $k_{\zeta}$ is not in $H^2)$, 
	we still can assert that $k_{\zeta}$ resides in 
	the \emph{Smirnov class} $N^+$ (quotients of bounded analytic functions with outer denominator, see
	Definition \ref{bounded-type}) \cite[Sect.~2.5]{Duren}.
	A little arithmetic tells us that
	\begin{align*}
		k_{\zeta} / \overline{k_{\zeta}}	
		&=	\frac{1 - \overline{\zeta}u}{1 - \overline{\zeta}z} \cdot 
			\frac{1 - \zeta\overline{z}}{1 - \zeta \overline{u}}  \\
		&=	\frac{1 - \overline{\zeta}u}{1 - \zeta \overline{u}} \cdot 
			\frac{1 - \zeta\overline{z}}{1 - \overline{\zeta}z}  \\ 
		&=	\overline{\zeta}u\left(\frac{\zeta \overline{u}-1 }{1 - \zeta \overline{u}}\right)
			\cdot  \zeta\overline{z} \left(	\frac{\overline{\zeta}z-1 }{1 - \overline{\zeta}z}\right)  \\  
		&=	\zeta \overline{\zeta}\overline{z}u\\
		&=\overline{z}u,
	\end{align*}
	which, when substituted into the equation $f = \overline{fz}u$, reveals that
	$$f/ k_{\zeta} = \overline{f / k_{\zeta}}.$$  Thus a function $f$ in $\K_u$
	satisfies $f = Cf$ if and only if $f/k_{\zeta}$ belongs to $N^+$ and is real almost everywhere on $\T$.

	\begin{Definition}
		A function $f$ belonging to the Smirnov class $N^+$ is called a 
		\emph{real Smirnov function} if its boundary function is real valued almost everywhere on $\T$. 
		Denote the set of all real Smirnov functions by $\mathcal{R}^+$.
	\end{Definition}
	
	The following elegant theorem of Helson \cite{Helson} provides an explicit formula for
	real Smirnov functions.

	\begin{Theorem}
		If $f \in \mathcal{R}^+$, then there exist inner functions $\psi_1, \psi_2$ such that
		\begin{equation}\label{eq:Helson}
			f = i \frac{\psi_1 + \psi_2}{\psi_1 - \psi_2} 
		\end{equation}
		and $\psi_1 - \psi_2$ is outer.
	\end{Theorem}

	\begin{proof}
		The function $$\tau(z) = i\frac{1+z}{1-z}$$ maps $\D$ onto the upper half-plane.
		Thus $\tau^{-1} \circ f$ is of bounded characteristic (i.e., a quotient of $H^{\infty}$ functions) and is unimodular almost everywhere on $\T$.
		Thus $\tau^{-1} \circ f = \psi_1 / \psi_2$ is a quotient of inner functions and hence
		$f$ has the desired form.
	\end{proof}

	Since each $f$ in $\K_u$ can be expressed in the form $(\alpha + i \beta) k_{\zeta}$ where $\alpha$ and $\beta$
	belong to $\mathcal{R}^+$, it follows from Helson's Theorem that any function in $\K_u$ can be written
	in a simple algebraic manner using only a handful of inner functions.  In particular, pseudocontinuations
	arise directly through the mechanism of Schwarz reflection (i.e., through pseudocontinuations of inner functions).

	Unfortunately, the Helson representation \eqref{eq:Helson} can be difficult to compute with.  For instance,
	it is a somewhat messy calculation to find a Helson representation for the \emph{K\"obe function} 
	\begin{equation*}
		k(z) = \frac{z}{(1-z)^2} = z +2z^2 + 3z^3 + \cdots,
	\end{equation*}
	which maps $\D$ bijectively onto $\C \backslash (-\infty,-\frac{1}{4}]$.  We therefore propose a more constructive
	description of the functions in $\R^+$, which originates in \cite{ROF, SCSSC}.
	First note that if $f = I_f F$ is the inner-outer factorization of a function $\mathcal{R}^+$, then
	\begin{equation*}
		I_f F =  \left[\frac{-4 I_f}{(1 - I_f)^2} \right]  \cdot \left[  \frac{(1 - I_f)^2 F}{-4} \right],
	\end{equation*}
	where the first term is in $\mathcal{R}^+$ and has the same inner factor as $f$
	and where the second is \emph{outer} and in $\mathcal{R}^+$.
	Thus to describe functions in $\mathcal{R}^+$, it suffices to describe \emph{real outer (RO) functions}.
	An infinite product expansion, in terms of Cayley-like transforms of inner functions, is obtained in \cite{ROF}.
	We refer the reader there for further details.

\section{Aleksandrov-Clark measures}\label{Section:Clark}
	The study of Aleksandrov-Clark measures dates back to the seminal work of Clark \cite{MR0301534} 
	and it continued with deep work of Aleksandrov \cite{MR931885, MR1039571, MR1734326} and Poltoratski \cite{MR1223178}.
	Since then, Aleksandrov-Clark measures have appeared in the study of spectral theory, composition operators, 
	completeness problems, and mathematical physics. Several sources for this important topic, 
	including references to their applications and the connections mentioned above, are \cite{CMR, MR2198367, MR2394657}.

\subsection{Clark measures} 
	If $\mu$ is a probability measure on $\T$ then the function
	\begin{equation*}
		F_{\mu}(z) := \int_{\T} \frac{\zeta + z}{\zeta - z} \,d \mu(\zeta)
	\end{equation*}
	is analytic on $\D$ and satisfies $F_{\mu}(0) = 1$. Furthermore, a simple computation shows that 
	\begin{equation*}
		\Re F_{\mu}(z) = \int_{\T} \frac{1 - |z|^2}{|\zeta - z|^2} \,d \mu(\zeta)
	\end{equation*}
	is positive on $\D$.  The following classical result of Herglotz asserts that this process can be reversed \cite[p.~10]{Duren}.

	\begin{Theorem}[Herglotz]\label{Herglotz}
		If $F$ is analytic on $\D$, $F(0) = 1$, and $\Re F > 0$ on $\D$, 
		then there is a unique probability measure $\mu$ on $\T$ such that $F = F_{\mu}$. 
	\end{Theorem}

	If $u$ is inner and $u(0) = 0$, then for each $\alpha$ in $\T$ the function 
	\begin{equation*}
		\frac{1 + \overline{\alpha} u(z)}{1 - \overline{\alpha} u(z)}
	\end{equation*}
	satisfies the hypothesis of Herglotz's theorem.  Indeed, this function is clearly analytic on $\D$,
	evaluates to $1$ at $z = 0$, and satisfies
	\begin{equation*}
		\Re\left(\frac{1 + \overline{\alpha} u(z)}{1 - \overline{\alpha} u(z)}\right) 
		= \frac{1 - |u(z)|^2}{|\alpha - u(z)|^2} > 0.
	\end{equation*}
	By Herglotz' Theorem, there is a unique probability measure $\sigma_{\alpha}$ on $\T$ such that 
	\begin{equation} \label{Clark-meas-defn}
		\frac{1 + \overline{\alpha} u(z)}{1 - \overline{\alpha} u(z)} = \int \frac{\zeta + z}{\zeta - z} \,d\sigma_{\alpha}(\zeta).
	\end{equation}
	The resulting family of measures 
	\begin{equation*}
		\{\sigma_{\alpha}: \alpha \in \T\}
	\end{equation*}
	are called the \emph{Clark measures corresponding to $u$} (see Subsection \ref{SubsectionAC} for their generalization,
	the \emph{Aleksandrov-Clark measures}).
	The following proposition
	summarizes some of their basic properties (see \cite{CMR} for details). 

	\begin{Proposition}\label{sigma-facts}
		For an inner function $u$ satisfying $u(0) = 0$, the corresponding family of Clark measures $\{\sigma_{\alpha}: \alpha \in \T\}$
		satisfy the following.
		\begin{enumerate}\addtolength{\itemsep}{0.5\baselineskip}
			\item $\sigma_{\alpha} \perp m$ for all $\alpha$.
			\item $\sigma_{\alpha} \perp \sigma_{\beta}$ for $\alpha \not = \beta$. 
			\item $\sigma_{\alpha}$ has a point mass at $\zeta$ if and only if $u(\zeta) = \alpha$ and $|u'(\zeta)| < \infty$. Furthermore 
				\begin{equation*}
					\sigma_{\alpha}(\{\zeta\}) = \frac{1}{|u'(\zeta)|}.
				\end{equation*}
			
			\item A carrier\footnote{By a {\em carrier} for $\sigma_{\alpha}$ we mean a Borel set $C\subset \T$ for which $\sigma_{\alpha}(A\cap C)=\sigma_{\alpha}(A)$ for all Borel subsets $A\subset\T$. Carriers are not unique and a carrier is often different from a support. For example, if a measure consists of a dense set of point masses on $\T$, then a carrier would consist of just these point masses, whereas the support would be $\T$. }
				for $\sigma_{\alpha}$ is the set 
				\begin{equation*}
					\left\{\zeta \in \T: \lim_{r \to 1^{-}} u(r \zeta) = \alpha\right\}.
				\end{equation*}
		\end{enumerate}
	\end{Proposition}

	This process can also be reversed.  Starting with a singular probability measure $\mu$ on $\T$ (e.g., $\mu = \delta_1$)
	one forms the \emph{Herglotz integral}
	\begin{equation*}
		\qquad H_{\mu}(z) := \int_{\T} \frac{\zeta + z}{\zeta - z} d \mu(\zeta), \quad z \in \D.
	\end{equation*}
	This function has positive real part on $\D$ and satisfies $H(0) = 1$. Thus
	\begin{equation*}
	\qquad u_{\mu}(z) := \frac{H(z) - 1}{H(z) + 1} \tag{$z \in \D$}
	\end{equation*}
satisfies 
$$\Re\left(\frac{1 + u_{\mu}(z)}{1 - u_{\mu}(z)}\right) = \frac{1 - |u_{\mu}(z)|^2}{|1 - u_{\mu}(z)|^2} = \int_{\T} P_{z}(\zeta) d\mu(\zeta).$$
Using the fact that 
$$\lim_{r \to 1^{-}} \int_{\T} P_{r w}(\zeta) d \mu(\zeta) = 0$$
for $m$-almost every $w \in \T$ (since $\mu \perp m$ -- see \eqref{Poisson-sym}) we see that $u_{\mu}$ is inner and $\mu$ is the Clark measure with $\alpha = 1$ corresponding to $u_{\mu}$. 

\subsection{Clark unitary operator}

For an inner function $u$ with $u(0) = 0$, recall that the compressed shift is the operator 
$S_u: \K_u \to \K_u$ defined by $S_u f = P_u(z f)$. For each $\alpha \in \T$ define 
\begin{equation}\label{CUO}
U_{\alpha}: \K_u \to \K_u, \quad U_{\alpha} := S_u + \alpha \Big(1 \otimes \frac{u}{z} \Big),
\end{equation}
where $1 \otimes \frac{u}{z}$ denotes the rank one operator satisfying $(1 \otimes \frac{u}{z})f = \inner{f, \frac{u}{z}}1$
(note that since $u(0) = 0$ the constant function $1$ belongs to $\K_u$). The main theorem here is the following (see \cite{CMR} for details).

\begin{Theorem}[Clark \cite{MR0301534}]\label{Clark-U-T}
For an inner function $u$ with $u(0) = 0$, the operator \eqref{CUO}
is a cyclic unitary operator whose spectral measure is carried by the Borel set 
$$\left\{\zeta \in \T: \lim_{r \to 1^{-}} u(r \zeta) = \alpha\right\}.$$
The eigenvalues of $U_{\alpha}$ are the points $\zeta$ in $\T$ so that $u(\zeta) = \alpha$ and $|u'(\zeta)| < \infty$. 
The corresponding eigenvectors are the boundary kernels $k_{\zeta}$. 
\end{Theorem}

\begin{Example}\label{Clark-ex}
	Consider the model space $\K_{z^3} = \bigvee\{1,z,z^2\}$.
	Letting $S_3 := S_{z^3}$, we see that
	\begin{equation*}
		S_3 1 = z, \quad S_3 z = z^2, \quad S_3 z^2 = 0,
	\end{equation*}
	so that the matrix representation of $S_3$ with respect to the basis $\{1, z, z^2\}$ is 
	\begin{equation*}
		\begin{bmatrix}
		 0 & 0 & 0 \\
		 1 & 0 & 0 \\
		 0 & 1 & 0
		\end{bmatrix}.
	\end{equation*}
	The Clark unitary operator in this case is
	$$U_{\alpha} = S_3 + \alpha (1 \otimes z^2).$$ 
	Similarly, since
	$$U_{\alpha} 1 = z, \quad U_{\alpha} z = z^2, \quad U_{\alpha} z^2 = \alpha,$$
	the matrix representation of $U_{\alpha}$ with respect to this basis is 
	\begin{equation*}
		\begin{bmatrix}
		 0 & 0 & \alpha  \\
		 1 & 0 & 0 \\
		 0 & 1 & 0
		\end{bmatrix}.
	\end{equation*}
	The unitarity of $U_{\alpha}$ follows easily from the fact that $|\alpha| = 1$.
	The eigenvalues of $U_{\alpha}$ are precisely the solutions to $u(z) = \alpha$
	(i.e., the three cube roots of $\alpha$) and corresponding eigenvectors are
	the boundary kernels
	\begin{equation*}
		k_{\beta}(z) = \frac{1 - \overline{\beta}^3 z^3}{1 - \overline{\beta} z} = 1 + \overline{\beta} z + \overline{\beta}^2 z^2,
	\end{equation*}
	where $\beta^3 = \alpha$.  Each eigenvector $k_{\beta}$ has norm $\sqrt{3}$, which is precisely the square root of the 
	modulus of the angular derivative (namely $3z^2$) of $z^3$ at the point $\beta$ in $\T$.   
	Moreover, the normalized eigenvectors of $U_{\alpha}$
	form an orthonormal basis for $\K_{z^3}$ (as expected since $U_{\alpha}$ is a unitary operator on a finite dimensional space).
\end{Example}

The Clark theory also gives us a concrete spectral representation for $U_{\alpha}$ via the Clark measure $\sigma_{\alpha}$. 

\begin{Theorem}[Clark \cite{MR0301534}]
For an inner function $u$ with $u(0) =0$, let $\sigma_{\alpha}$ be the unique finite positive Borel measure on $\T$ satisfying 
$$\frac{1 + \overline{\alpha} u(z)}{1 - \overline{\alpha} u(z)} = \int \frac{\zeta + z}{\zeta - z} \,d\sigma_{\alpha}(\zeta).$$ The operator 
$$(V_{\alpha} f)(z) = (1 - \overline{\alpha} u(z)) \int \frac{f(\zeta)}{1 - \overline{\zeta} z} \,d\sigma_{\alpha}(\zeta)$$ is a unitary operator from $L^2(\sigma_{\alpha})$ onto $\K_u$. Furthermore, 
if 
$$Z_{\alpha}: L^2(\sigma_{\alpha}) \to L^2(\sigma_{\alpha}), \quad (Z_{\alpha} f)(\zeta) = \zeta f(\zeta),$$
then $V_{\alpha} Z_{\alpha} = U_{\alpha} V_{\alpha}$.
\end{Theorem}

\begin{Example}
Returning to Example \ref{Clark-ex}. Let $\zeta_1, \zeta_2, \zeta_3$ be the three solutions to $u(z) = \alpha$ (i.e., to $z^3 = \alpha$). Then the discrete measure 
$$d \sigma_{\alpha}  = \tfrac{1}{3} \delta_{\zeta1} + \tfrac{1}{3}\delta_{\zeta_2} + \tfrac{1}{3} \delta_{\zeta_3}$$ 
satisfies 
$$\frac{1 + \overline{\alpha} z^3}{1 - \overline{\alpha} z^3} = \int \frac{\zeta + z}{\zeta - z} \,d\sigma_{\alpha}(\zeta).$$
From the theorem above, 
$$V_{\alpha}: L^2(\sigma_{\alpha}) \to \K_{z^3}, 
\quad (V_{\alpha} f)(z) = (1 - \overline{\alpha} z^3) \int \frac{f(\zeta)}{1 - \overline{\zeta}} \,d\sigma_{\alpha}(\zeta).$$ 
To see this worked out, note that if $$f = c_1 \chi_{\zeta_1} + c_2 \chi_{\zeta_2} + c_3 \chi_{\zeta_3,}$$ 
a typical element of $L^2(\sigma_{\alpha})$, then 
$$(V f)(z) = (1 - \overline{\alpha} z^3) \left( c_1 \frac{1/3}{1 - \overline{\zeta_1} z}  +  c_2  \frac{1/3}{1 - \overline{\zeta_2} z}  + c_3  \frac{1/3}{1 - \overline{\zeta_3} z} \right).$$ Since $\zeta_{j}^{3} = \alpha$, one can verify that the expression above is a polynomial of degree at most $2$, which are precisely the elements of $\K_{z^3}$. 
Furthermore, $\{\sqrt{3} \chi_{\zeta_1}, \sqrt{3} \chi_{\zeta_2}, \sqrt{3} \chi_{\zeta_3}\}$ is an orthonormal basis for $L^2(\sigma_{\alpha})$ and 
$$(V_{\alpha} \sqrt{3} \chi_{\zeta_j})(z) = \frac{\sqrt{3}}{3} k_{\zeta_j}(z).$$ But since 
$$\left\{\frac{\sqrt{3}}{3} k_{\zeta_j}(z): j = 1, 2, 3\right\}$$ is an orthonormal basis for $\K_{z^3}$ we see that 
$V_{\alpha}$ is unitary. From here one can easily see, by basic linear algebra, that  $V_{\alpha}$ intertwines $U_{\alpha}$ with $Z_{\alpha}$. 
\end{Example}

The adjoint of the unitary operator $V_{\alpha}$ is of particular interest. We know that $V_{\alpha} f$ belongs to $\K_{u}$ for every $f$ in $L^2(\sigma_{\alpha})$ and, as such, $V_{\alpha} f$ has a radial limit almost everywhere with respect to Lebesgue measure $m$. The following theorem of Poltoratski \cite{MR1223178} says much more. 

\begin{Theorem}[Poltoratski \cite{MR1223178}]\label{Polt-V}
For $f \in L^2(\sigma_{\alpha})$, 
$$\lim_{r \to 1^{-}} (V_{\alpha} f)(r \zeta) = f(\zeta)$$
for $\sigma_{\alpha}$-almost every $\zeta \in \T$. 
\end{Theorem}

The significance of this result is that it says that functions in model spaces (i.e., the functions $V_{\alpha} f$ for $f \in L^2(\sigma_{\alpha})$ in the previous theorem) have nontangential limits on finer sets than generic functions from $H^2$. We have seen a result along these lines already in Theorem \ref{TheoremAC}. The previous theorem says a bit more in that all functions from $\K_u$ have non-tangential limits $\sigma_{\alpha}$ almost everywhere. This is especially notable since $\sigma_{\alpha} \perp m$ and a classical result \cite{Lusin, Piranian} says that if $E$ is any closed subset of $\T$ with Lebesgue measure zero, then there exists an $f$ in $H^2$ (which can be taken to be inner) whose radial limits do not exists on $E$. For example, if $\sigma_{\alpha}$ has an isolated point pass at $\zeta$, then every function in $\K_u$ has a nontangential limit at $\zeta$. In this particular case, the result should not be surprising since, by Proposition \ref{sigma-facts}, $u$ will have a finite angular derivative at $\zeta$ and so by the Ahern-Clark result (Theorem \ref{TheoremAC}) every function in $\K_u$ has a nontangential limit at $\zeta$.

Here is one more fascinating and useful gem about Clark measures due to Aleksandrov \cite{MR931885} (see also \cite{CMR}).  

\begin{Theorem}[Aleksandrov \cite{MR931885}]\label{ADC}
Let $u$ be an inner function with associated family of Clark measures $\{\sigma_{\alpha}: \alpha \in \T\}$.  If $f \in C(\T)$, then
$$\int_{\T} \left(\int_{\T} f(\zeta) d\sigma_{\alpha}(\zeta)\right) dm(\alpha) = \int_{\T} f(\xi)\, dm(\xi).$$
\end{Theorem}

\begin{proof}
For a fixed $z \in \D$ notice that if 
$P_{z}(\zeta)$ is the Poisson kernel, then taking real parts of both sides of \eqref{Clark-meas-defn} we get 
\begin{align*}
\int_{\T} \left( \int_{\T} P_{z}(\zeta) d\sigma_{\alpha}(\zeta)\right) dm(\alpha) & = \int_{\T} \left( \frac{1 - |u(z)|^2}{|\alpha - u(z)|^2}\right) dm(\alpha)\\
& = \int_{\T} P_{u(z)}(\alpha) dm(\alpha)\\
& = 1\\
& = \int_{\T} P_{z}(\zeta) dm(\zeta).
\end{align*}
Thus the theorem works for finite linear combinations of Poisson kernels. Using the fact that finite linear combinations of Poisson kernels are dense in $C(\T)$ one can use a limiting argument to get the general result\footnote{See \cite[Ch.~9]{CMR} for the details on the density of the linear span of the Poisson kernels in $C(\T)$ as well as the completion of this limiting argument}.
\end{proof}

\begin{Remark}
Aleksandrov \cite{MR931885}
showed that $C(\T)$ can be replaced by $L^1$ in the preceding disintegration theorem. There are a few technical issues to work out here. For example,  the inner integrals
$$\int f(\zeta) d\sigma_{\alpha}(\zeta)$$
 in the disintegration formula do not seem to be well defined for $L^1$ functions since indeed the measure $\sigma_{\alpha}$ can contain point masses on $\T$ while $L^1$ functions are defined merely  $m$-almost everywhere. However, amazingly, the function 
 $$\alpha \mapsto \int f(\zeta) d\sigma_{\alpha}(\zeta)$$
 is defined for $m$-almost every $\alpha$ and is integrable. An argument with the monotone class theorem is used to prove this more general result. See also \cite{CMR} for the details. 
\end{Remark}

\subsection{Aleksandrov-Clark measures}\label{SubsectionAC} 
The alert reader will have noticed that the title of this section was Aleksandrov-Clark measures but we seem to have only discussed {\em Clark} measures. As it turns out, one can still examine 
the equation 
$$\frac{1 + \overline{\alpha} u(z)}{1 - \overline{\alpha} u(z)} = \int \frac{\zeta + z}{\zeta - z} d\sigma_{\alpha}(\zeta)$$
but where $u$ belongs to $H^{\infty}$ and $\|u\|_{\infty} \leqslant 1$, but $u$ is not necessarily inner. The measures $\sigma_{\alpha}$ will no longer be singular but they do form an interesting class of measures explored by Aleksandrov and others (see \cite{CMR} for references). One can identify all the parts of these measures (absolutely continuous, singular continuous, point masses, etc.) as well as develop a decomposition theorem as before. Versions of these measures appear in mathematical physics through some papers of B. Simon and T. Wolff \cite{Simon-Wolff, Simon}.

\section{Completeness problems}\label{CP}

In Subsection \ref{Subsection:Riesz}, we discussed circumstances under which a sequence 
	$\{k_{\lambda}\}_{\lambda \in \Lambda}$ of kernel functions forms a Riesz basis for $\K_u$ 
	(see Theorem \ref{Theorem:Interpolation}).  More generally, one can ask the following question:
	If $\Lambda$ is a sequence in $\D$, when does
	\begin{equation*}
		K_{\Lambda} := \bigvee_{\lambda \in \Lambda} k_{\lambda} = \K_u?
	\end{equation*}
Such problems are known as {\em completeness problems}. 

	By considering orthogonal complements and the reproducing property of $k_{\lambda}$, we see that $K_{\Lambda} = \K_u$ precisely when there are no nonzero 
	functions in $\K_u$ that vanish on $\Lambda$.  
	We have already seen that $K_{\Lambda} = \K_u$ if $\Lambda$ is not a Blaschke sequence or if $\Lambda$ has an accumulation point 
	in $\D$ (Proposition \ref{P-density-easy}).  A little further thought shows that the same holds 
	if $\Lambda$ has an accumulation point in $\T \setminus \sigma(u)$ since otherwise there would exist a nonzero function in $\K_u$
	whose analytic continuation vanishes on a set having a limit point in its domain (Proposition \ref{P-analytic-continuation}).	For general Blaschke sequences, the situation is more complicated. However, we can rephrase the completeness 
	problem in terms of Toeplitz kernels\footnote{Kernels of Toeplitz operators have been studied before and we refer the reader to the 
	papers \cite{MR1736197, CBSTK, MR1300218} for further details.}. 

	\begin{Proposition}
		If $\Lambda$ is a Blaschke sequence in $\D$ and 
		$B$ is the Blaschke product corresponding to $\Lambda$, then for any inner function $u$ we have 
		\begin{equation*}
			K_{\Lambda} \not = \K_u \quad\iff\quad \ker T_{\overline{u} B} \not = \{0\}.
		\end{equation*}
	\end{Proposition}

	\begin{proof}
		Note that if $K_{\Lambda} \neq \K_u$, then there is a nonzero function $f$ in $\K_u$ that vanishes on $\Lambda$. 
		This happens precisely when $f = g B$ for some $g$ in $H^2$ that is not identically zero. 
		But since $f$ belongs to $\K_u$ we know that $\overline{u} f$ is contained in $\overline{z H^2}$. 
		This means that $\overline{u} B g$ belongs to $\overline{z H^2}$. Now apply the Riesz projection 
		(of $L^2$ onto $H^2$) to $\overline{u} B g$ to see that $T_{\overline{u} B}g = P(\overline{u}Bg) = 0$. 
		For the other direction, simply reverse the argument. 
	\end{proof}

	Clark's approach in \cite{MR0301534} to the completeness problem is based on the following related approximation problem of Paley-Wiener.
	Suppose that $\{x_n\}_{n \geqslant 1}$ is an orthonormal basis for a Hilbert space $\mathcal{H}$ and $\{y_n\}_{n \geqslant 1}$ is a sequence in $\mathcal{H}$. 
	If these sequences are {\em close} in some way, does $\{y_n\}_{n \geqslant 1}$ span $\mathcal{H}$? There are various results which say this is often the case (see \cite[Sec.~86]{MR0071727} or \cite{N1} for a survey of these results) with various interpretations of the term {\em close}. We will see a specific example of this below. 

Clark's approach to the completeness problem $K_{\Lambda} = \K_u$ was to find an orthonormal basis for $\K_u$ consisting of (normalized) boundary kernel functions. One way to accomplish this is to find a unitary operator $U$ on $\K_u$ whose spectrum consists only of eigenvalues (pure point spectrum) and such that the eigenvectors of $U$ are boundary kernel functions. But indeed we have already seen this can be done. From Clark's theorem (Theorem \ref{Clark-U-T}) we see that $U_{\alpha}$ is unitary and, under the right spectral circumstances (i.e., the corresponding Clark measure $\sigma_{\alpha}$ is purely atomic), its normalized eigenvectors
\begin{equation*}
\left\{\frac{k_{\zeta}}{\sqrt{|u'(\zeta)|}}: u(\zeta) = \alpha, |u'(\zeta)| < \infty\right\}
\end{equation*}
form an orthonormal basis for $\K_{u}$. From here one can apply a Paley-Wiener type result to obtain completeness results for the family $\{k_{\lambda}: \lambda \in \Lambda\}$. Let us give two illustrative examples from the original Clark paper \cite{MR0301534}. 

\begin{Example} Suppose that $u$ is an inner function and $\alpha \in \T$ such that the corresponding Clark measure $\sigma_{\alpha}$ is discrete, i.e., 
$$\sigma_{\alpha} = \sum_{n = 1}^{\infty} \frac{1}{|u'(\zeta_n)|} \delta_{\zeta_n}.$$
Note the use of Proposition \ref{sigma-facts} here where $\{\zeta_n\}_{n \geqslant 1}$ are the solutions to $u = \alpha$. Since $\sigma_{\alpha}$ is a discrete measure and $U_{\alpha}$ is unitary, the normalized eigenvectors 
$$h_{\zeta_n} := \frac{k_{\zeta_n}}{\sqrt{|u'(\zeta_n)|}}$$
are an orthonormal basis for $\K_u$.  For our given sequence 
$\Lambda = \{\lambda_n\}_{n \geqslant 1} \subset \D$, let 
$$h_{\lambda_n} := \frac{k_{\lambda_n}}{\|k_{\lambda_n}\|}$$
be the normalized kernel functions. If the sequence $\{\lambda_n\}_{n \geqslant 1}$ satisfies 
$$\sum_{n = 1}^{\infty} \|h_{\zeta_n} - h_{\lambda_n}\|^{2} < 1,$$ then a generalization of a theorem of Paley and Wiener (see \cite[Sec.~86]{MR0071727}) says that $\{h_{\lambda_{n}}\}_{n \geqslant 1}$ forms an unconditional basis for $\K_u$, meaning that every $f$ in $\K_u$ has an expansion $f = \sum_{n} a_n h_{\lambda_n}$. 
\end{Example}

This next example brings in some earlier work of Sarason \cite{MR0192355}. 

\begin{Example}
Consider the model space $\K_u$ with inner function 
$$u(z) := \exp\left(\frac{z + 1}{z - 1}\right).$$ The Clark measure $\sigma_{1}$ corresponding to $u$ has the set $\{u = 1\}$ as its carrier, which turns out to be the discrete set 
\begin{equation*}
\zeta_n := \frac{1 + 2 \pi i n}{1 - 2 \pi i n}. \tag{$n \in \Z$}
\end{equation*}
In \cite{MR0192355} it is shown that the operator 
$$[W f](z) :=  \frac{\sqrt{2}}{1 - z} \int_{0}^{1} f(t) u^{t}(z) dt$$
defines unitary operator which maps $\K_u$ onto $L^2[0, 1]$. In our current setting, the real usefulness of this operator comes from the formula 
$$W(e^{i \gamma t}) = q k_{\lambda}, \quad \gamma := i \frac{1 + \overline{\lambda}}{1 - \overline{\lambda}},$$
where $|\lambda| \leqslant 1$ and $q$ is a constant. In particular, this means that $W$ maps each orthonormal basis element $e^{2 \pi i n t}$ for $L^2[0, 1]$ to a constant times $k_{\zeta_n}$. So, for a given a sequence $\{\lambda_n\}_{n \in \Z} \subset \D$ satisfying  
$$\max_{n \in \Z} |\gamma_n - 2 \pi n| < \frac{\log 2}{ \pi},$$
a theorem of Paley and Wiener (see \cite[Sec.~86]{MR0071727}) says that $\{e^{i \gamma_n t}\}_{n \in \Z}$ forms an unconditional basis for $L^{2}[0, 1]$. Under this criterion, we get that the normalized kernels $\{k_{\lambda_n}/\|k_{\lambda_n}\|: n \in \Z\}$ form an unconditional basis for $\K_u$. 

We will consider other unitary operators that transfer orthonormal bases of $L^2$ to kernel functions in the next section. 
\end{Example}

We also point out a recent result of Baranov and Dyakonov \cite{BD} that also discusses when a perturbation of Clark bases form Riesz bases.

\section{Model spaces for the upper-half plane}

\subsection{The Hardy space again} Let $$\C_{+} := \{z \in \C: \Im z > 0 \}$$ denote the upper half plane and let $\HH^{2}$ denote the {\em Hardy space of the upper half plane}. These are the analytic functions $f$ on $\C_{+}$ for which 
$$\sup_{y > 0} \int_{-\infty}^{\infty} |f(x + i y)|^2 \,dx < \infty.$$
The reader will immediately notice the analogue of the ``bounded integral mean'' condition \eqref{eq:HpNorm} from the Hardy space $H^2$ on the unit disk. As it turns out, the majority of the theory from $H^2$ (on the disk) transfers over {\em mutatis mutandis} to $\HH^2$. For example, every $f$ in $\HH^2$ has an almost everywhere well-defined `'radial'' boundary function 
$$f(x) := \lim_{y \to 0^{+}} f(x + i y)$$ and this function belongs to $L^2(dx):= L^2(\R,dx)$. Moreover, 
$$\int_{-\infty}^{\infty} |f(x)|^2 dx = \sup_{y > 0} \int_{-\infty}^{\infty} |f(x + i y)|^2 dx.$$ 
This allows us to endow $\HH^2$ with an inner product 
$$\langle f, g \rangle := \int_{-\infty}^{\infty} f(x) \overline{g(x)} dx,$$
where, in a manner that is analogous to the $H^2$ case, $f(x)$ and $g(x)$ are the almost everywhere defined boundary functions of $f, g$ in 
$\HH^2$. 
Still further, equating $H^2$ with the functions in $L^2(\T,m)$ whose negatively indexed Fourier coefficients vanish, 
we have the corresponding characterization
\begin{equation*}
\HH^2 := \big\{\widehat{f}: f \in L^2(0, \infty)\big\},
\end{equation*}
where $$\widehat{f}(t) = \int_{-\infty}^{\infty} f(x) e^{-2 \pi i x t} \,dx$$ is the {\em Fourier transform} of $f$. 

The reason that most of the Hardy space theory on the disk can be imported into $\HH^2$ is because of the fact that the operator 
$$\mathcal{U}: L^2(m) \to L^2(dx), \quad (\mathcal{U} f)(x) := \frac{1}{\sqrt{\pi} (x + i)}{f(w(x))},$$
where 
$$w(z) := \frac{z - i}{z + i}$$ is a conformal map from $\C_{+}$ onto $\D$, is a unitary map from $L^2(\T,m)$ onto $L^2(\R,dx)$. 
In particular, $\mathcal{U}$ maps $H^2$ unitarily onto $\HH^2$ and maps $L^{\infty}(\T,m)$ isometrically onto $L^{\infty}(\R,dx)$. 

\subsection{Inner functions} 
We say that an analytic function $\Theta$ on $\C_{+}$ is {\em inner} if $|\Theta(z)| \leqslant 1$ for all $z \in \C_{+}$ and if
the almost everywhere defined boundary function $\Theta(x)$ is unimodular almost everywhere on $\R$. 
The two most basic types of inner functions on $\C_+$ are 
\begin{equation*}
	S^{c}(z) := e^{i c z} \tag{$c > 0$}
\end{equation*}
and
\begin{equation*} 
b_{\lambda}(z) := \frac{z - \lambda}{\,z - \overline{\lambda}\,}. \tag{$\lambda \in \C_{+}$}
\end{equation*}
The first class of examples are the most basic type of {\em singular inner function} on $\C_+$ and the second type 
is the most basic type of {\em Blaschke product}. 

Further examples of singular inner functions are given by 
$$S_{\mu, c}(z) := e^{i c z} \exp\left(-\frac{1}{\pi i} \int_{-\infty}^{\infty}\left(\frac{1}{x - z}  - \frac{x}{1 + x^2}\right)d \mu(x)\right),$$
where $c > 0$ and $\mu$ is a positive measure on $\R$ that is singular with respect to Lebesgue measure and 
is {\em Poisson finite}, in the sense that
$$\int_{-\infty}^{\infty} \frac{1}{1 + x^2} d\mu(x) < \infty.$$
The general Blaschke product is formed by specifying its zeros 
$$\Lambda = \{\lambda_{n}\}_{n \geqslant 1} \subset \C_{+} \backslash \{i\}$$ and forming the product 
$$B_{\Lambda}(z) = b_{i}(z)^{m}\prod_{n \geqslant 1} \epsilon_{n} b_{\lambda_n}(z),$$ where the constants 
$$\epsilon_n := \frac{|\lambda_{n}^{2} + 1|}{\lambda_{n}^{2} + 1}$$ are chosen so that $\epsilon_n b_{\lambda_n}(i) > 0$. It is well known \cite{Garnett} that the product above converges uniformly on compact subsets of $\C_{+}$ to an inner function if and only if the zeros $\{\lambda_n\}_{n \geqslant 1}$ satisfy the \emph{Blaschke condition}
$$\sum_{n \geqslant 1} \frac{\Im \lambda_n}{1 + |\lambda_n|^2} < \infty.$$ 
Every inner function $\Theta$ can be factored uniquely as 
\begin{equation*}
\Theta = e^{i \gamma} B_{\Lambda} S_{\mu, c} \tag{$\gamma \in [0, 2 \pi)$}
\end{equation*}

\subsection{Model spaces} 
For an inner function $\Theta$ define a {\em model space} $\KK_{\Theta}$ on the upper-half plane as 
$$\KK_{\Theta} := \HH^2 \ominus \Theta \HH^2.$$ 
Much of the theory for the model spaces $\K_u$ on the disk carries over to the upper-half plane. 
For instance, via Proposition \ref{Prop:fgzu}, we can regard $\KK_{\Theta}$ as a space of boundary functions on $\R$ by noting that
\begin{equation} \label{K-line}
\KK_{\Theta} := \HH^2 \cap \Theta \overline{\HH^2}.
\end{equation}  

An inner function $\Theta$ has an analytic continuation across $\R \setminus \sigma(\Theta)$, where 
$$\sigma(\Theta) := \left\{z \in \C_{+}^{-}: \liminf_{\lambda \to z} |\Theta(\lambda)| = 0\right\}$$
is the {\em spectrum} of $\Theta$, which turns to be the union of the closure the zeros of $\Theta$ along with the support of the singular measure for $\Theta$.  Every $f$ in $\KK_{\Theta}$ has an analytic continuation across $\R \backslash \sigma(\Theta)$. Using the identity \eqref{K-line}, we see that for each $f$ in $\KK_{\Theta}$ there is a corresponding $g$ in $\HH^2$ for which $f = \overline{g}\Theta$ almost everywhere on $\R$. As was done with \eqref{PC-H-D} this allows us to obtain a formula 
\begin{equation*}
\widetilde{f}(z) := \frac{\overline{g(\overline{z})}}{\,\overline{\Theta(\overline{z})}\,} \tag{$z \in \C_{-}$}
\end{equation*}
for a pseudocontinuation of $f$ (to the lower half-plane $\C_{-}$) of bounded type. 

Of particular importance here are the reproducing kernel functions 
\begin{equation*}
K_{\lambda}(z) = \frac{i}{2 \pi} \frac{1 - \overline{\Theta(\lambda)} \Theta(z)}{z - \overline{\lambda}} \tag{$\lambda, z \in \C_{+}$}
\end{equation*}
 These belong to $\KK_{\Theta}$ and satisfy 
 \begin{equation*}
f(\lambda) = \langle f, K_{\lambda}\rangle \tag{$\lambda \in \C_{+}$}
\end{equation*}

\subsection{Clark theory again} 

As to be expected, there is a Clark theory for $\KK_{\Theta}$. Indeed, for an inner function $\Theta$, the function 
$$m = i \frac{1 + \Theta}{1 - \Theta}$$ is analytic on $\C_{+}$ and satisfies
$\Im m \geqslant 0$ there (i.e., $m$ is a \emph{Herglotz function}).  In this setting, 
the Herglotz representation theorem guarantees the existence of
parameters $b \geqslant 0, c \in \R$, and a positive Poisson finite measure $\mu_{\Theta}$ on $\R$ so that 
$$m(\lambda) = b z + c + \frac{1}{\pi} \int \left(\frac{1}{x - \lambda} - \frac{x}{1 + x^2}\right) d\mu_{\Theta}(x).$$ 
Using Poltoratski's result (Theorem \ref{Polt-V}) we can define the operator 
\begin{equation} \label{Q-T}
Q_{\Theta}: \KK_{\Theta} \to L^2(\mu_{\Theta}), \quad Q f = f|C_{\Theta},
\end{equation}
where $C_{\Theta}$ is a carrier for $\Theta$. 
This operator turns out to be unitary. 

As a final remark, the alert reader is probably getting \emph{d\'ej\`a vu} when looking at the function $m$ and its associated inner function $\Theta$. The current discussion is probably reminding the reader of our earlier treatment of Clark measures. Indeed, if one begins with an inner $\Theta$ and looks at the family of inner functions 
$$\{\overline{\alpha} \Theta: \alpha \in \T\},$$
there is an associated family of Poisson finite measures 
$$\{\mu_{\overline{\alpha} \Theta}: \alpha \in \T\},$$
the Clark measures associated with $\Theta$. 
 Many of the properties of Clark measures that hold for inner $u$ on $\D$ have direct analogs for inner functions $\Theta$ on $\C_{+}$.

\subsection{Weyl-Titchmarsh inner functions}\label{Subsection:WT}

In this section we point out a relationship between Schr{\" o}dinger operators and model spaces. In order not to get too deep into technical details, we will be a little vague about proper definitions. The reader looking for precision should consult \cite{MP05, Tes}. 

For a real potential $q$ in $L^2(a, b)$, define the \emph{Schr{\" o}dinger operator}
$$u \mapsto u'' + q u.$$ If one imposes the boundary condition
$$u(b) \cos \phi + u'(b) \sin \phi = 0,$$ this operator becomes a densely defined, self-adjoint operator on $L^2(a, b)$. 

For each $\lambda$ in $\C$ the differential equation $u'' + q u = \lambda u$ has a has a non-trivial solution
 $u_{\lambda}$ and one can form the {\em Weyl-Titchmarsh} function (at $a$) by 
$$m(\lambda) := \frac{u_{\lambda}'(a)}{u_{\lambda}(a)}.$$ 
One can show that the function $\Theta$ defined on $\C_{+}$ by 
$$\Theta := \frac{m - i}{m + i}$$
is inner. 

For each $\lambda \in \C_{+}$ define  
$$w_{\lambda}(x) := \frac{u_{\lambda}(x)}{u_{\lambda}'(a) + i u_{\lambda}(a)}$$ and notice that $w_{\lambda}$ is also a solution to the Schr{\" o}dinger equation $u'' + q u = \lambda u$, the so called {\em normalized solution}. Now define the following transform $W$ on $L^2(a, b)$ (sometimes called the {\em modified Fourier transform}) by 
\begin{equation*}
[W f](\lambda) := \int_{a}^{b} f(x) w_{\lambda}(x) dx \tag{$\lambda \in \C$}
\end{equation*}
 Schr{\" o}dinger operators and model spaces are related by means of the following theorem from \cite{MP05}.

\begin{Theorem}[Makarov-Poltoratski]
The modified Fourier transform $W$ is, up to a factor of $\sqrt{\pi}$, a unitary operator from $L^{2}(a, b)$ onto $\KK_{\Theta}$. Furthermore, $W \overline{w_\lambda} = \pi K_{\lambda}$.
\end{Theorem}

With the preceding theorem one can recast the completeness problems for solutions of the Schr{\" o}dinger equation
in $L^2(a, b)$ in terms of the completeness problem for reproducing kernels in $\KK_{\Theta}$. Secondly, the unitary operator $W Q_{\Theta}: L^{2}(a, b) \to L^2(\mu_{\Theta})$ yields a spectral representation of the Schr{\" o}dinger operator $u \mapsto u'' + q u$ in that it is unitarily equivalent to the densely defined operator multiplication by the independent variable on $L^2(\mu_{\Theta})$.
\section{Generalizations of model spaces}

\subsection{Model spaces in $H^p$} 
	One can also consider model subspaces of $H^p$ for $p \neq 2$.
	Recall that if $0 < p < \infty$, then $H^p$ is the space of analytic functions $f$ on $\D$ for which 
	$$\|f\|_{p} := \lim_{r \to 1^{-}} \left(\int_{\T} |f(r \zeta)|^2 dm(\zeta)\right)^{1/p}$$ 
	is finite.  If $p \geqslant 1$, then $H^p$ is a Banach space while if $0 < p < 1$, then $H^p$ is a topological vector space. 
	Almost all of the basic function-theoretic results that hold for $H^2$ functions carry over to $H^p$ (e.g., each $H^p$ has
	an almost everywhere defined nontangential boundary function that belongs to $L^p$).
	
	For $p \neq 2$, the spaces $H^p$ are not Hilbert spaces, although they do have readily identifiable duals:
	\begin{itemize}\addtolength{\itemsep}{0.5\baselineskip}
		\item For $1 < p < \infty$, the dual space $(H^p)^{*}$ can be identified with $H^q$, where $q$ is the H{\" o}lder conjugate index to $p$
			(i.e., $\frac{1}{p}+\frac{1}{q} = 1$).
		\item For $p=1$, the dual space $(H^1)^{*}$ can be identified with the space of all analytic functions 
			of bounded mean oscillation (BMOA),
		\item For $0 < p < 1$, $(H^p)^{*}$ can be identified with certain classes of smooth functions on $\D^{-}$ (Lipschitz or Zygmund classes).
	\end{itemize}
	In each of these cases, the dual pairing is given by the {\em Cauchy pairing}
	\begin{equation}\label{eq:IP}
		\lim_{r \to 1^{-}} \int_{\T} f(r \zeta) \overline{g(r \zeta)} \,dm(\zeta)
	\end{equation}
	where $f \in H^p$ and $g \in (H^p)^{*}$.
	Standard references are \cite{Duren, Garnett}. 

The unilateral shift $S: H^p \to H^p$ defined as usual by $S f = z f$ is continuous on $H^p$
and, for $1 \leqslant p < \infty$, the invariant subspaces of $S$ are still the Beurling-type invariant subspaces $u H^p$ for some inner $u$.
The invariant subspaces of the backward shift\footnote{Much of the literature here abuses notation here and often uses the notation $S^{*} f = (f - f(0))/z$ for the backward shit on $H^p$ even though, technically speaking, the adjoint $S^{*}$ is defined on $(H^p)^{*}$ and not $H^p$.}
$$B f = \frac{f - f(0)}{z}$$ on 
$H^p$ are known to be 
\begin{equation*}
\K_{u}^{p} := H^p \cap u \overline{z H^p} \tag{$1 \leqslant p < \infty$}
\end{equation*}
For $1 < p < \infty$,
the proof relies on using the duality relationship between $(H^p)^{*}$ and $H^q$ via the integral pairing \eqref{eq:IP} to compute the annihilator of $u H^q$. For $p = 1$, the backward shift invariant invariant subspaces are still the same (i.e., 
$\K_u^{1} = H^1 \cap u \overline{z H^1}$),
except that the proof has to deal with some complications from the more difficult dual space of $H^1$. When $0 < p < 1$, the backward shift invariant subspaces are too complicated to describe in this survey. A full accounting all of these results are found in \cite{CR}. Most of the function-theoretic results of this survey for model spaces $\K_u$ in $H^2$ can be restated appropriately for $\K_u^{p}$ spaces -- usually with nearly the same proof. 

\subsection{deBranges-Rovnyak spaces}
There is an important generalization of model spaces, the \emph{deBranges-Rovnyak spaces}, that play an increasingly important role in analysis. 
Unlike the model spaces $\K_{u}$, which are parameterized by inner functions $u$, these spaces are parameterized by functions $u$ in $H^{\infty}_{1}$, the unit ball in $H^{\infty}$ (i.e., $u \in H^{\infty}, \|u\|_{\infty} \leqslant 1$). These spaces often have similar properties as model spaces but they also have many differences. We will not go into great detail here but just point out the existence of these types of spaces and their very basic properties. 
The standard reference for this subject is Sarason's book \cite{MR1289670} from 1994, 
although a new book by Fricain and Mashreghi should appear soon.

For $u \in H^{\infty}_{1}$ (not necessarily inner!), define the kernel function 
\begin{equation*}
k_{\lambda}(z) := \frac{1 - \overline{u(\lambda)} u(z)}{1 - \overline{\lambda} z} .\tag{$\lambda, z \in \D$}
\end{equation*}
 The author will recognize this kernel, when $u$ is an inner function, as the reproducing kernel for the model space $\K_{u}$, but here $u$ is not necessarily inner. We construct a reproducing kernel Hilbert space $\HH(u)$ of analytic functions on $\D$ from this kernel in the following way. First we notice that this kernel is a positive in the sense that
$$\sum_{1 \leqslant l, j \leqslant n} c_{j} \overline{c_l} k_{\lambda_j}(\lambda_l) \geqslant 0$$ for every set of constants $c_1, \ldots, c_n$ and points $\lambda_1, \ldots, \lambda_n$ in $\D$. We initially populate our space $\HH(u)$ with finite linear combinations of kernel functions and define a norm on these functions in a way that makes the $k_{\lambda}$ the reproducing kernels for $\HH(u)$. We do this by defining  
$$\left \|\sum_{j = 1}^{n} c_j k_{\lambda_j}\right\|^2 = \left\langle \sum_{j = 1}^{n} c_j k_{\lambda_j}, \sum_{l = 1}^{n} c_l k_{\lambda_l} \right\rangle := \sum_{1 \leqslant l, j \leqslant n} c_{j} \overline{c_l} k_{\lambda_j}(\lambda_l).$$ One can check that this is a norm and that the corresponding inner product on the vector space of finite linear combinations of kernel functions, make this vector space a pre-Hilbert space. The deBranges-Rovnyak space $\HH(u)$ is the closure of this pre-Hilbert space under this norm.  A more standard way of defining $\mathscr{H}(u)$ in terms of Toeplitz operators is described in \cite{MR1289670}. 

Using the positivity of the kernel, one can prove that $\HH(u)$ is reproducing kernel Hilbert space that is contractively contained in $H^2$. When $\|u\|_{\infty} < 1$, then $\HH(u) = H^2$ with an equivalent norm. On the other hand when $u$ is inner, then $\HH(u) = \K_u$ with equality of norms.

Recall that the compressed shift $S_u$ on $\K_u$ serves as a model for certain types of contraction operators. It turns out that $\HH(u)$ spaces can be used to model other types of contractions where the condition (i) of Theorem \ref{NF-Thm} is somewhat relaxed. 

\section{Truncated Toeplitz operators}
	We do not attempt to give a complete overview of the rapidly 
	developing theory of truncated Toeplitz operators.
	The recent survey article \cite{GR-TTO} and the seminal article \cite{Sarason} of Sarason should be consulted
	for a more thorough treatment.	
	We content ourselves here with a brief summary of the principal definitions and basic results.

	Recall from Subsection \ref{SubsectionBasic} that the Toeplitz operator $T_{\phi}:H^2\to H^2$ with symbol
	$\phi$ in $L^{\infty}$ is defined by $P(\varphi f)$, where $P$ is the orthogonal projection from $L^2$ onto $H^2$. 
	One can show that $T_{\varphi}$ is a bounded operator on $H^2$ that satisfies $T_{\phi}^* = T_{\overline{\phi}}$
	and $\norm{T_{\phi}} = \norm{\phi}_{\infty}$.
	The $(j, k)$ entry of the matrix representation of $T_{\phi}$ with respect to the orthonormal normal basis $\{1, z, z^2, \ldots\}$ of $H^2$ 
	is $\widehat{\varphi}(k - j)$, which yields an infinite Toeplitz matrix (constant along the diagonals). 
	
	An old result of Brown and Halmos \cite{BH} gives a convenient algebraic characterization of Toeplitz operators:
	a bounded operator $T$ on $H^2$ is a Toeplitz operator if and only 
	\begin{equation}\label{eq:BH}
		T = S T S^{*}, 
	\end{equation}
	where $S$ is the unilateral shift on $H^2$.
	Many other things are known about Toeplitz operators: their spectrum, essential spectrum, commutator ideals, 
	algebras generated by certain collections of Toeplitz operators, etc. 
	These topics are all thoroughly discussed in \cite{MR2223704, MR2270722, MR0361893}.

	It is well-known that the commutant 
	\begin{equation*}
		\{S\}' := \{A \in \mathcal{B}(H^2): S A = A S\}
	\end{equation*}
	of the unilateral shift $S$ is equal to the set
	\begin{equation*}
		\{T_{\phi}: \phi \in H^{\infty}\}
	\end{equation*}
	of all \emph{analytic} Toeplitz operators.  This naturally leads one to consider the commutant
	$\{S_u\}'$ of the compressed shift $S_u$.  The answer, due to Sarason \cite{Sarason-NF}, 
	prompts us to consider the following.
	
	\begin{Definition}
		If $u$ is inner and $\phi \in L^{\infty}$, then the \emph{truncated Toeplitz operator} (TTO) with symbol $\phi$
		is the operator $A_{\phi}^u:\K_u\to\K_u$ defined by $A^{u}_{\phi} f = P_u (\phi f)$, where
		$P_u$ is the orthogonal projection \eqref{Pu-def} of $L^2$ onto $\K_u$. 
	\end{Definition}

	\begin{Theorem}[Sarason]
		For inner $u$, 
		\begin{equation*}
			\{S_{u}\}' = \{A^{u}_{\phi}: \phi \in H^{\infty}\}.
		\end{equation*}
	\end{Theorem}

	This theorem, which initiated the general study of \emph{commutant lifting theorems}, gives an initial
	impetus for the study of truncated Toeplitz operators.  However, to obtain the full class of truncated Toeplitz operators, 
	one needs to consider symbols in $L^2$, as opposed to $L^{\infty}$. For $\phi$ in $L^2$, one can 
	define an operator $A_{\phi}^u$ on $\K_u \cap H^{\infty}$ (which is dense in $\K_u$ via Proposition \ref{P-H-dense}) by 
	\begin{equation*}
		A^u_{\phi} f := P_{u}(\varphi f).
	\end{equation*}
	If $A_{\phi}^u$ can be extended to a bounded operator on all of $\K_u$, then we say $A_{\phi}^u$ 
	is a bounded truncated Toeplitz operator.  When there is no chance of confusion, we write $A_{\phi}$
	instead of $A_{\phi}^u$.  We let $\mathcal{T}_{u}$ denote 
	the set of all bounded truncated Toeplitz operators on $\K_u$, remarking that $\mathcal{T}_u$ is a
	weakly closed linear space. 

	There are some superficial similarities between truncated Toeplitz operators and Toeplitz operators. 
	For instance, one has the adjoint formula $A_{\varphi}^{*} = A_{\overline{\varphi}}$ and an analogue of the Brown-Halmos result
	\eqref{eq:BH}:  a bounded operator $A$ on $\K_u$ is a truncated Toeplitz operator if and only if 
	\begin{equation*}
		A_z A A_z^* = A + R,
	\end{equation*}
	where $R$ is a certain operator on $\K_u$ of rank $\leqslant 2$. However, the differences are greater than the similarities.
	For instance, the symbol of a Toeplitz operator is unique, yet for truncated Toeplitz operators
	$A_{\varphi} = A_{\psi}$ if and and only if $$\varphi - \psi \in u H^2 + \overline{u H^2}.$$  In particular,
	the inequality $\|A_{\varphi}\| \leqslant \|\varphi\|_{\infty}$ is frequently strict.
	While there are no compact Toeplitz operators, many compact TTOs exist (e.g., any TTO on $\K_u$ if $\dim\K_u < \infty$). Finally, not every bounded truncated Toeplitz operator can be represented with a bounded symbol.
	
	Recall from Section \ref{SectionConjugation} that $Cf = \overline{fz}u$ defines a conjugation on $\K_u$.
	It turns out that each truncated Toeplitz operator is \emph{complex symmetric}, in the sense that
	$A_{\phi} = CA_{\phi}^*C$.  More generally, we say that a bounded operator $T$ on a complex Hilbert space $\h$
	is a complex symmetric operator (CSO) 
	if there exists a conjugation $C$ on $\h$ so that $T = CT^*C$ \cite{G-P,G-P-II,CSPI, SNCSO,CCO}.
	The class of complex symmetric operators is surprisingly large and, somewhat surprisingly, many CSOs
	can be shown to be unitarily equivalent to truncated Toeplitz operators \cite{GR-TTO, TTOSIUES, STZ, UETTOAS}.
	Clarifying the precise relationship between CSOs and TTOs is an ongoing effort, chronicled to some extent in \cite{GR-TTO}.

	Finally, we mention the following beautiful connection between Clark's unitary operators (Section \ref{Section:Clark})
	and truncated Toeplitz operators.
	
	\begin{Theorem}
		If $\phi \in L^{\infty}$ and $u$ is inner, then 
		$$A_{\phi}^u = \int \varphi(U_{\alpha}) \,dm(\alpha),$$
		in the sense that
		\begin{equation*}
		\langle A_{\phi}^u f, g \rangle  
		= \int \langle \phi(U_{\alpha}) f, g\rangle \,dm(\alpha),
		\end{equation*}
		for $f,g$ in $\K_u$.
	\end{Theorem}

\section{Things we did not mention}
In this final section we point the reader in the direction of other aspects of model spaces
that we do not have the time to fully explore.

\subsection{Carleson measures} Carleson in \cite{LC-58, LC-62} (see also \cite{Garnett})
characterized those measures $\mu$ on $\D$ for which the inclusion operator from $H^2$ to $L^2(\mu)$ is continuous, i.e., 
\begin{equation*}
\int_{\D} |f|^2 d \mu \leqslant C \|f\|^2 \tag{$f \in H^2$}
\end{equation*}
where $C > 0$ is a positive number independent of $f$. This result has been generalized to many other Hilbert spaces of analytic functions. What are the Carleson measures for the model spaces? A discussion of these measures along with historical references can be found in \cite{BFGHR}. 

\subsection{Vector-valued model spaces} 
One can form the vector-valued Hardy space $H^2$ of $\C^n$-valued analytic functions on $\D$ for which the integral means 
$$\int_{\T} \|f(r \zeta)\|^{2}_{\C^n} dm(\zeta)$$ are bounded as $r\to1^-$.  For a matrix valued inner function $U$ on $\D$ (i.e., a matrix-valued analytic function $U(z)$ on $\D$ that satisfies $\|U\| \leqslant 1$ on $\D$ and which is unitary valued a.e.~on $\T$), we can define the $\C^n$-valued model space $H^2 \ominus U H^2$. The compression of the shift to these model spaces are the model operators for contractions via a similar Sz.-Nagy-Foia\c{s} theory for higher defect indices. Some well-known sources for all this operator theory are \cite{Bercovici, N1,N3, RR}. There is even a Clark theory for these vector-valued model spaces including a version of the Aleksandrov disintegration theorem (Theorem \ref{ADC})\cite{MR2652521, MR3079829}.

\subsection{Hankel operators}
There are strong connections between model spaces and Hankel operators.
The standard texts on the subject are \cite{Peller, N2,N3}.

\subsection{Extremal problems} There is a nice connection between model spaces, truncated Toeplitz operators, and classical extremal problems. For a rational function $\psi$ whose poles are in $\D$, a classical problem in complex analysis is to compute the quantity 
\begin{equation}\label{extremal-P}
\sup_{f \in H^{1}_{1}} \left| \frac{1}{2 \pi i} \oint_{\T} \psi(\zeta) f(\zeta) d \zeta\right|,
\end{equation}
along with the functions $f$ for which the above supremum is achieved. In the above definition, $H^{1}_{1}$ denotes the unit ball in the Hardy space $H^1$. These types of problems and extensions of them, where $\psi$ belongs to $L^{\infty}$ and is not assumed to be rational, have been studied since the early twentieth century. We refer the reader to \cite{NLEPHS} for a survey of these classical results. In that same paper is the following result:

\begin{Theorem} 
If $\psi$ is a rational function whose poles are contained in $\D$ then
\begin{equation} \label{Ross-Garcia}
\sup_{f \in H^{1}_{1}} \left| \frac{1}{2 \pi i} \oint_{\T} \psi(\zeta) f(\zeta) d \zeta\right| = \sup_{g \in H^{2}_{2}}  \left| \frac{1}{2 \pi i} \oint_{\T} \psi(\zeta) g^2(\zeta) d \zeta\right|.
\end{equation}
Furthermore, if $u$ is the finite Blaschke product whose zeros are precisely those poles with corresponding multiplicities and if $\phi := u \psi$, then 
$$\sup_{f \in H^{1}_{1}} \left| \frac{1}{2 \pi i} \oint_{\T} \psi(\zeta) f(\zeta) d \zeta\right| = \|A^{u}_{\phi}\|,$$
where $A^{u}_{\phi}$ is the analytic truncated Toeplitz operator on $\K_u$ with symbol $\phi$. 
\end{Theorem}

In fact, the linear and quadratic supremums in \eqref{Ross-Garcia} are the same for general $\psi$ in $L^{\infty}(\T)$ \cite{MR2597679}.

\bibliography{ModelSurvey} 

\end{document}